\documentclass[11 pt]{amsart}

\usepackage[latin1]{inputenc}
\usepackage{amsmath,amsthm,amssymb,amscd,graphicx,psfrag}
\usepackage[colorlinks=true, pdfstartview=FitV, linkcolor=blue, citecolor=red, urlcolor=blue]{hyperref}

\newtheorem{thm}{Theorem}[section]
 \newtheorem{cor}[thm]{Corollary}
 \newtheorem{lem}[thm]{Lemma}
 \newtheorem{claim}[thm]{Claim}
 \newtheorem{prop}[thm]{Proposition}
 
 \theoremstyle{definition}
 \newtheorem{df}[thm]{Definition}
 \theoremstyle{remark}
 \newtheorem{rem}[thm]{Remark}
 \newtheorem{ex}{Example}
 \numberwithin{equation}{section}

\def\be#1 {\begin{equation} \label{#1}}
\newcommand{\ee}{\end{equation}}

\newcommand{\mb}{\medskip\noindent}
\newcommand{\gb}{\bigskip\noindent}
\newcommand{\R}{\mathbb R}
\newcommand{\M}{\mathcal M}

\newcommand{\N}{\mathbb N}

\newcommand{\s}{\mathcal S}
\newcommand{\F}{\mathcal F}

\def \dsp {\displaystyle}
\def \vsp {\vspace{6pt}}

\begin{document}

\author{Fr\'ed\'eric Bernicot}
\address{Fr\'ed\'eric Bernicot - CNRS - Universit\'e Lille 1 \\ Laboratoire de math\'ematiques Paul Painlev\'e \\ 59655 Villeneuve d'Ascq Cedex, France}
\curraddr{}
\email{frederic.bernicot@math.univ-lille1.fr}

\author{Pierre Germain}
\address{Pierre Germain - Courant Institute of Mathematical Sciences \\ New York University \\ 251 Mercer Street \\ New York, N.Y. 10012-1185 \\ USA}
\email{pgermain@cims.nyu.edu}

\title{Boundedness of bilinear multipliers whose symbols have a narrow support}



\date{\today}

\begin{abstract} This work is devoted to studying the boundedness on Lebesgue spaces of bilinear multipliers on $\R$ whose symbol is narrowly supported around a curve (in the frequency plane). We are looking for the optimal decay rate (depending on the width of this support) for exponents satisfying a sub-H\"older scaling. As expected, the geometry of the curve plays an important role, which is described. \\
This has applications for the bilinear Bochner-Riesz problem (in particular, boundedness of multipliers whose symbol is the characteristic function of a set), as well as for the bilinear restriction-extension problem.
\end{abstract}

\maketitle

\begin{quote}
\footnotesize\tableofcontents
\end{quote}

\section{Introduction}

\subsection{The central question}

Pseudo-products were introduced by Bony \cite{bony} and Coifman-Meyer~\cite{coifmanmeyer}; we shall define, and study them, only in the case of bilinear operators.
Given a symbol $m(\xi,\eta)$, the pseudo-product $B_m$, acting on functions over $\mathbb{R}^d$, is given by
\begin{align*}
B_m(f,g)(x) & := \frac{1}{(2 \pi)^{d/2}} \int_{\mathbb{R}^d} \int_{\mathbb{R}^d} e^{ix(\xi+\eta)} m(\xi,\eta) \widehat{f}(\xi) \widehat{g}(\eta)\,d\eta\,d\xi \\
& = \mathcal{F}^{-1} \left[\xi \rightarrow \int_{\mathbb{R}} m(\xi-\eta,\eta) \widehat{f}(\xi-\eta) \widehat{g}(\eta)\,d\eta\right](x)
\end{align*}
(notations, in particular the convention used for the Fourier transform, are given in Section~\ref{notationsandpreliminaries}).
Our aim is to study, for $d=1$, the connection between singularities of $m$ (on various scales, and along any possible smooth geometry) and the boundedness properties of $B_m$. 

A relevant model is the following: consider a smooth curve $\Gamma \subset \mathbb{R}^2$, and let $m_\epsilon$ be a symbol, of size less than $1$, and $0$ at a distance $\epsilon$ of $\Gamma$. What about the regularity of $m_\epsilon$? A first possibility is to simply ask that it varies on a typical length $\epsilon$; one can also ask more smoothness in the direction tangential to $\Gamma$: see definitions~\ref{tatou} and~\ref{herisson} for more precise definitions.

\bigskip

\mb {\bf Question:} {\it Set $d=1$. For which Lebesgue exponents $(p,q,r) \in [1,\infty]^3$ and for which functions $\alpha(\epsilon)$ does there hold
\begin{equation}
\label{ecureuil} \left\| B_{m_\epsilon} \right\|_{L^p\times L^q \rightarrow L^{r'}} \lesssim \alpha(\epsilon)
\end{equation}
or, to put it in a more symmetrical fashion,
\begin{equation}
\left| \langle B_{m_\epsilon}(f,g) \,,\,h \rangle \right| \lesssim \alpha(\epsilon) \left\|f \right\|_{L^p} \left\|f \right\|_{L^q} \left\|f \right\|_{L^r} \;\;?
\end{equation}
(Typically, $\alpha(\epsilon)$ will be a power function with perhaps a logarithmic correction).}

\bigskip

As we will see, answering to the above question contributes to solving the following problems:
\begin{itemize}
\item (Bilinear Bochner-Riesz) Given a compact domain $K$ with a smooth boundary, for which $p,q,r,\kappa$ is $B_{m^\kappa_K}$ bounded from $L^p \times L^q$ to $L^{r'}$, if $m^\kappa_K(\eta,\xi) = \chi_K(\xi,\eta) \operatorname{dist} ((\xi,\eta),\partial K)^{\kappa}$?
\item (Bilinear restriction-extension) Given a curve $\Gamma$, for which $p,q,r$ is $B_{d\sigma_\Gamma}$ bounded from $L^p \times L^q$ to $L^{r'}$?
\end{itemize}
(the notations $\chi_K$ and $d\sigma_\Gamma$ are defined in Section~\ref{notationsandpreliminaries}).

\subsection{Analogy with the linear case}

The above questions have clear analogs in the linear case: these are the well-known Bochner-Riesz (boundedness between Lebesgue spaces of $f \mapsto \mathcal{F}^{-1} \left[ \widehat{f} \chi_K \operatorname{dist} (\cdot , \partial K)^{\kappa}\right]$ for $K$ subset of $\mathbb{R}^d$), restriction (boundedness of $f \mapsto \widehat{f}|_\Gamma$ for $\Gamma$ hypersurface of $\mathbb{R}^d$) and extension (boundedness of $f \mapsto \widehat{f d\sigma_\Gamma}$) problems. Notice that combining the restriction and extension problem gives the transformation $f \mapsto \widehat{\widehat{f} d\sigma_\Gamma}$, which we called restriction-extension in the bilinear setting.

The case of dimension 1 only requires very standard harmonic analysis. The case of dimension 2, which we now address, is more subtle, and the geometry (of $K$, $\Gamma$) starts playing an important role.

Let us discuss first the Bochner-Riesz problem. By Plancherel, it is clear that Fourier multipliers with symbol $\chi_K \operatorname{dist} (\cdot , \partial K)^{\kappa}$ will be bounded on $L^2$. If $K$ is a polygon, they will be bounded on $L^p$ by appealing to the one-dimensional theory. It came as a surprise that if the boundary of $K$ is curved, the corresponding multiplier is only bounded on $L^p$ if $p=2$: this is the content of the Fefferman ball multiplier theorem~\cite{Fefferman}. We emphasize that for the linear theory, a key dichotomy is flat versus curved boundary of $\partial K$. Boundedness of Fourier multipliers with symbols $(|\xi|^2-1)^\kappa_+$, with $\kappa>0$, was characterized by Carleson and Sj\"olin~\cite{CS}, with different proofs proposed by Fefferman~\cite{Fefferman2} and Cordoba~\cite{Cordoba}.

The restriction and extension problems are simply dual to one another. Here again, the key distinction is $\Gamma$ flat versus $\Gamma$ curved: if $\Gamma$ is flat, no restriction theorem holds (ie $f \mapsto \widehat{f}|_\Gamma$ is never bounded), but if $\Gamma$ is curved, restriction properties can be proved. The full restriction theorem for the circle is due to Fefferman and Stein~\cite{Fefferman2}, with a different proof proposed by C\'ordoba~\cite{Cordoba2}.

Let us notice here that many of the above proofs rely on the study of {\it linear} operators with symbols of the type described above: $m_\epsilon(\eta,\xi)$.

The Bochner-Riesz or restriction problems in higher dimension $d \geq 3$ are still open, we refer to the monographs by Stein~\cite{Stein} and Grafakos~\cite{Grafakos} for an introduction.

\subsection{Known results for the bilinear case}

Much of the research on bilinear operators has focused on boundedness between Lebesgue spaces at the H\"older scaling: from $L^p \times L^q$ to $L^{r'}$, with $\frac{1}{p} + \frac{1}{q} = \frac{1}{r}$.

The first results, obtained by Coifman and Meyer~\cite{coifmanmeyer}, allowed for a singularity localized at a point: the symbol $m$ satisfies a Mikhlin-type condition $\left|\partial_\xi^\alpha \partial_\eta^\beta m(\xi,\eta) \right| \lesssim \frac{1}{(|\xi|+|\eta|)^{|\alpha| + |\beta|}}$. Then $B_m$ is bounded from $L^p \times L^q$ to $L^{r'}$ if $1=\frac{1}{p} + \frac{1}{q} + \frac{1}{r}$, and $1<p,q,r<\infty$. For another result on boundedness with a singularity at a point, see Gustafson, Nakanishi and Tsai~\cite{GNT}, and the version of their result in Guo and Pausader~\cite{GP}.

The bilinear Hilbert transform corresponds to taking $d=1$, and for $m$ the characteristic function of a (perhaps tilted) half-plane: $m$ is singular along a line. The celebrated results of Lacey and Thiele \cite{LT2, LT1, LT3, LT4} gave boundedness of $B_m$ from $L^p \times L^q$ to $L^{r'}$ with $1<p,q<\infty$ and $0<\frac{1}{r'}=\frac{1}{p}+\frac{1}{q}<\frac{3}{2}$. These results were later extended by Grafakos and Li \cite{GL0, Li} to cover the case where $m$ is the characteristic function of a polygon. We refer the reader to \cite{B}, where the first author proved boundedness for particular square functions built on a covering of the frequency plane with polygons.

Finally, let us discuss the case where $m$ is the characteristic function of the ball: the singularity is now localized on a curved set. Diestel and Grafakos proved  that the characteristic function of the four-dimensional ball is not a bounded bilinear multiplier operator from $L^p(\R^2) \times L^q (\R^2 )$ into $L^{r'} (\R^2 )$ outside the local $L^2$-case, i.e. when $1/p + 1/q + 1/r=1$ and $2 \leq p, q, r < \infty$ fails. Conversely, it was shown by Grafakos and Li \cite{GL} that the characteristic function of the unit disc in $\R^2$ is a bounded bilinear multiplier on $L^p(\R^2) \times L^q (\R^2 )$ into $L^{r'} (\R^2 )$ in the local $L^2$-case.  The corresponding problem in higher dimension remains unresolved. The positive results of boundedness can be extended to ellipses. 

Let us point out that the easiest case, i.e. for $p=q=r=2$ can be directly studied with Plancherel inequality. This very particular setting has been involved with the use of  $X^{s,b}$ spaces,  and also has been extended for less regular symbols to the so-called ``multilinear convolution in $L^2$'' by Tao in \cite{tao}.

\subsection{Obtained results}

While the previously cited works deal with the critical (and more complicated) case where the exponents $p,q,r\in[1,\infty]$ satisfy the H\"older scaling
$ 1=\frac{1}{r}+\frac{1}{p}+\frac{1}{q}$, and sometimes allow for Lebesgue exponents less than one, we shall also assume in this whole article that
$$
1 \leq p,q,r \leq \infty,
$$
and extend attention to the ``sub-critical'' range
\begin{equation}
\label{eq:subholder}
1\leq \frac{1}{r}+\frac{1}{p}+\frac{1}{q}.
\end{equation}

\bigskip

What are the important geometric features of $\Gamma$ as far as boundedness of $B_{m_\epsilon}$ is concerned? As will be illustrated in the following theorems, the crucial point is actually whether $\Gamma$ has tangents parallel to the axes $\{ \xi = 0 \}$, $\{ \eta = 0 \}$, and $\{ \xi+\eta = 0 \}$ axes.

\begin{df}
\label{def:noncar}
Given a smooth curve $\Gamma$ in the $(\xi,\eta)$ plane, its characteristic points are the points where its tangent is parallel to the $\{ \xi = 0 \}$, $\{ \eta = 0 \}$, or $\{ \xi+\eta = 0 \}$ axes. We call $\Gamma$ characteristic (respectively, non-characteristic) if such points exist (respectively, do not exist).
\end{df}

The best bounds for $B_{m_\epsilon}$ are obtained if $\Gamma$ is non-characteristic; the next best thing is when the set of characteristic points is finite, with non zero curvature of $\Gamma$; and the worst possibility is of course when a piece of $\Gamma$ is a segment parallel to one of the axes $\xi$, $\eta$, or $\xi+\eta$.

It is worth noticing that, as opposed to the linear case, the curvature of $\Gamma$ does not play any role {\it per se}, but only because a non zero curvature prevents the points close to a characteristic point of being too close to characteristic themselves. In particular, at non-characteristic points, it is indifferent whether $\Gamma$ has a curvature or not.

\bigskip

Before stating our theorems, let us define the regularity classes for $m_\epsilon$ which we will use.
The first class only requires $m_\epsilon$ to be supported in an $\epsilon$-neighborhood of $\Gamma$, with derivatives of order $\frac{1}{\epsilon}$.

\begin{df}
\label{tatou}
The scalar-valued symbol $m_\epsilon$ belongs to the class $\mathcal{M}_\epsilon^\Gamma$ if
\begin{itemize}
\item $\Gamma$ is a smooth curve in $\mathbb{R}^2$.
\item $m_\epsilon$ is supported in $B(0,1)$, as well as in a neighborhood of size $\epsilon$ of $\Gamma$.
\item $\left|\partial_\xi^\alpha \partial_\eta^\beta m_\epsilon (\xi,\eta)\right| \lesssim \epsilon^{-|\alpha|-|\beta|}$ for sufficiently many indices $\alpha$ and $\beta$.
\end{itemize}
\end{df}

The above class turns out to be too weak in the following case: $\Gamma$ characteristic, with a non-vanishing curvature, and nearly H\"older exponents. Then more tangential smoothness is required: this is the point of the next definition (which could be weakened a lot, but appropriate conditions would then become too technical).

\begin{df} 
\label{herisson} Close to $\Gamma$, it is possible to define ``normal directions'' simply by prolonging the normals to $\Gamma$, and ``tangential directions'': these are lines whose tangents are everywhere orthogonal to normal directions. If $\nu$ is the distance function to $\Gamma$, $\nabla \nu$ can be considered as the direction of the local normal coordinate and $(\nabla \nu)^{\perp}$ as the direction of the local tangential coordinate. We are interested in symbols $m_\epsilon$ satisfying a nice behavior in the tangential directions given by $(\nabla \nu)^\perp$. For a vector $X$, define
$$
\partial_X := X \cdot \nabla.
$$
The scalar-valued symbol $m_\epsilon$ belongs to the class $\mathcal{N}_\epsilon^\Gamma$ if
\begin{itemize}
\item $\Gamma$ is a smooth curve in $\mathbb{R}^2$. 
\item $m_\epsilon$ is supported in $B(0,1)$, as well as in a neighborhood of size $\epsilon$ of $\Gamma$.
\item for sufficiently many indices $\alpha,\beta$
$$ \left|\partial_{\nabla \nu}^\alpha \partial_{(\nabla \nu)^\perp} ^\beta m_\epsilon(\xi,\eta)\right| \lesssim \epsilon^{-|\alpha|}.$$
\end{itemize}
\end{df}

We now come to the obtained results; though we do not always state them in an optimal way for the sake of concision and clarity. The interested reader can refer to sections~\ref{sec:L2}, \ref{SOPP}, \ref{CTHP}, and \ref{sec:combinaison}, where the precise statements are given. Moreover, Section \ref{sec:extensions} contains extensions to rough curves $\Gamma$ (as soon as they satisfy some rectifiability and Ahlfors regularity properties).

\subsubsection{The non-characteristic case}

\begin{thm}[Corollary \ref{cor:nonc}] 
\label{thnc}
Assume that $\Gamma$ is nowhere characteristic, and let $m_\epsilon$ belong to $\mathcal{M}^\Gamma_\epsilon$; consider $p,q,r$ in $(1,\infty)$. Then
\begin{itemize} 
\item  If $\left\{ \displaystyle \begin{array}{l} 
1 \leq \frac{1}{p}+\frac{1}{q}+\frac{1}{r} \leq 2 \vsp \\ 
\frac{1}{r}+\frac{1}{q} \leq \frac{3}{2} \vsp \\ 
\frac{1}{p}+\frac{1}{q} \leq \frac{3}{2} \vsp \\ 
\frac{1}{p}+\frac{1}{r} \leq \frac{3}{2}, \vsp 
\end{array} \right.$, then $\displaystyle \|B_{m_\epsilon}\|_{L^p \times L^q \rightarrow L^{r'}} \lesssim \epsilon^{\frac{1}{p}+\frac{1}{q}+\frac{1}{r}-1}$,
and this exponent of $\epsilon$ is optimal.
\medskip

\item If $\displaystyle \left\{ \begin{array}{l} 
q \geq 2 \geq p,r \vsp \\ 
\frac{1}{p}+\frac{1}{r} \geq \frac{3}{2}, \vsp \end{array} \right.$, then $\displaystyle
\|B_{m_\epsilon}\|_{L^p \times L^q \rightarrow L^{r'}} \lesssim \epsilon^{\frac{1}{q}+\frac{1}{2}}$.
\medskip

\item The above statement of course remains true if the indices $(p,q,r)$ are permuted.
\medskip

\item If $\displaystyle \left\{ \begin{array}{l} p,q,r \leq 2 \\ \frac{1}{p}+\frac{1}{q}+\frac{1}{r} \geq 2 \end{array} \right.$, then $\|B_{m_\epsilon}\|_{L^p \times L^q \rightarrow L^{r'}} \lesssim \epsilon$
and this bound is optimal.

\end{itemize}
\end{thm}

The three cases distinguished in the above theorem cover the full range $1 < p,q,r < \infty$. The bounds are optimal except in the second case. Optimality extends to symbols in the smoother class $\mathcal{N}^\Gamma_\epsilon$, and this will also be the case in the next theorems when optimality is stated.

The proof of this theorem can be found in Section~\ref{sec:combinaison}, where the interpolation between endpoint type results obtained in sections~\ref{sec:L2} and~\ref{SOPP} is performed. The optimality statements follow from Section~\ref{NCATSC}.

\subsubsection{The non-vanishing curvature case}

In our first theorem, we only assume that $m_\epsilon \in \mathcal{M}^\Gamma_\epsilon$.

\begin{thm} 
\label{thnvc}
Assume $\Gamma$ has non-vanishing curvature, and let $m_\epsilon$ belong to $\mathcal{M}^\Gamma_\epsilon$; consider $p,q,r$ in $(1,\infty)$. Then
\begin{itemize}
\item If $\displaystyle p,q,r \geq 2$, then $\displaystyle \|B_{m_\epsilon}\|_{L^p \times L^q \rightarrow L^{r'}} \lesssim \epsilon^{\frac{1}{p}+\frac{1}{q}+\frac{1}{r}-1}$
and this power of $\epsilon$ is optimal.
\medskip

\item If $\displaystyle \left\{ \begin{array}{l} q > 2 \vsp \\ p,r < 2 \vsp \end{array} \right.$, then $\displaystyle \|B_{m_\epsilon}\|_{L^p \times L^q \rightarrow L^{r'}} \lesssim \epsilon^{-\frac{1}{2}+\frac{1}{q}+\frac{1}{2}\left(\frac{1}{p}+\frac{1}{r} \right) - \delta}$
for any $\delta>0$. The above power of $\epsilon$ is optimal up to the additional $\delta$.

\medskip

\item The above statement remains true if the indices $(p,q,r)$ are permuted.

\medskip
\item If $\displaystyle \frac{1}{p}+\frac{1}{q}+\frac{1}{r} > \frac{5}{2}$, then $\|B_{m_\epsilon}\|_{L^p \times L^q \rightarrow L^{r'}} \lesssim \epsilon$ and this bound is optimal.

\medskip
\item If $\displaystyle \left\{ \begin{array}{l} p,q,r<2 \vsp \\ \frac{3}{2} \leq \frac{1}{p}+\frac{1}{q}+\frac{1}{r} \leq \frac{5}{2}, \vsp \end{array}\right.$, then $\displaystyle\|B_{m_\epsilon}\|_{L^p \times L^q \rightarrow L^{r'}} \lesssim \epsilon^{\frac{1}{2} \left( \frac{1}{p}+\frac{1}{q}+\frac{1}{r}-\frac{1}{2} -\delta\right)}.$ for any $\delta>0$.
\end{itemize}
\end{thm}

The above theorem only leaves out the case where exactly one of the three Lebesgue indices is less than 2. The bounds stated are optimal except in the last case above; see Section~\ref{bilinrestr} for some improvements in this direction.

In order to cover the remaining cases, more tangential regularity from $m_\epsilon$ is needed: we will now assume that it belongs to $\mathcal{N}^\Gamma_\epsilon$.

\begin{thm} Assume $\Gamma$ has non-vanishing curvature, and let $m_\epsilon$ belong to $\mathcal{N}^\Gamma_\epsilon$; consider $p,q,r$ in $(1,\infty)$.
\begin{itemize}
\item If $\displaystyle \left\{ \begin{array}{l} p < 2 \\ q>r>2 \\ \frac{1}{p}+\frac{1}{r}>1 \end{array} \right.$ then $\displaystyle \|B_{m_\epsilon}\|_{L^p \times L^q \rightarrow L^{r'}} \lesssim \epsilon^{-\frac{1}{2}+\frac{1}{q}+\frac{1}{2}\left(\frac{1}{p}+\frac{1}{r} \right)-\delta }$ for any $\delta>0$, and this is optimal up to the additional $\delta$.

\medskip
\item If $\displaystyle \left\{ \begin{array}{l} \frac{1}{p} + \frac{1}{r}<1 \\ \frac{1}{p}+\frac{1}{q}<1 \\ \frac{1}{p}+\frac{1}{q}+\frac{1}{r}>1 \end{array} \right.$, then $\displaystyle \|B_{m_\epsilon}\|_{L^p \times L^q \rightarrow L^{r'}} \lesssim \epsilon^{\frac{1}{p}+\frac{1}{q}+\frac{1}{r}-1-\delta}$ for any $\delta>0$, and this is optimal up to the additional $\delta$.

\item The above statements remain true if the indices $(p,q,r)$ are permuted.
\medskip
\end{itemize}
\end{thm}

The three cases distinguished above do not completely cover the range where exactly one of the three Lebesgue indices is less than 2. We refrained from giving bounds for any $(p,q,r)$ since the obtained formulas become too complicated.

As for Theorem~\ref{thnc}, the two theorems above are proved by interpolating between the results of sections~\ref{sec:L2},~\ref{SOPP} and~\ref{CTHP}, and the optimality statements follow from Section~\ref{NCATSC}. This procedure is however not detailed, since it is very similar to Theorem~\ref{thnc}.

\subsubsection{The general case} 

\begin{thm}[Theorem \ref{thm:thmgeneral}] Let $\Gamma$ be an arbitrary curve, and let $m_\epsilon$ belong to $\mathcal{M}^\Gamma_\epsilon$; consider $p,q,r$ in $(1,\infty)$. Then
\begin{itemize}
\item If $\displaystyle p,q,r\geq 2$ and verify (\ref{eq:subholder}), then $\displaystyle \|B_{m_\epsilon}\|_{L^p \times L^q \rightarrow L^{r'}} \lesssim \epsilon^{\frac{1}{p}+\frac{1}{q}+\frac{1}{r}-1}$
and this power of $\epsilon$ is optimal.
\medskip

\item If at least two of the three indices $(p,q,r)$ are smaller than $2$, then $\displaystyle \|B_{m_\epsilon}\|_{L^p \times L^q \rightarrow L^{r'}} \lesssim \epsilon^{\inf\left(\frac{1}{p} , \frac{1}{q}, \frac{1}{r} \right)}$, and this power of $\epsilon$ is optimal.

\end{itemize}
\end{thm}

Of all possible values for $(p,q,r)$, the previous theorem only leaves aside the case where exactly two of the three indices $(p,q,r)$ are greater than 2. Once again, it can be obtained by interpolating between the results of sections~\ref{sec:L2},~\ref{SOPP}, and using the optimality criteria of Section~\ref{NCATSC}.

\begin{rem}[H\"older case]
Overall, when do we get the expected bound of order 1 for $B_{m_\epsilon}$ when the exponents $1 < p,q,r < \infty$ satisfy $\frac{1}{p}+\frac{1}{q}+\frac{1}{r}=1$? In two cases: either when $\Gamma$ is non-characteristic, or when $p,q,r>2$. If $\Gamma$ has non-zero curvature and $m_\epsilon \in \mathcal{N}^\Gamma_\epsilon$, then we are able to prove that $B_{m_\epsilon}$ has a norm $\sim \epsilon^\rho$ from $L^p \times L^q \rightarrow L^{r'}$ with $\rho \rightarrow 0$ as $\frac{1}{p}+\frac{1}{q}+\frac{1}{r}$ approaches 1 from above. We believe that the techniques developed in this paper can lead to the same result for arbitrary curves $\Gamma$.
\end{rem}

\subsection{Applications}

\subsubsection{Bilinear restriction-extension inequalities}

\label{bilinrestr}

The limiting point of view (when $\epsilon$ goes to $0$) can be partially described in term of ``bilinear restriction-extension'' inequalities. It is said that a curve $\Gamma$ satisfies such inequalities for exponents $(p,q,r) \in[1,\infty]^3$ if for every smooth, compactly supported function $\lambda$, the bilinear multiplier $B_m$ is bounded from $L^p \times L^q$ into $L^{r'}$ with $m=\lambda d\sigma_\Gamma$ and $d\sigma_\Gamma$ the arc-length measure on $\Gamma$ (carried on this curve).

It is easy to see that if (\ref{ecureuil}) holds for $p,q,r$ with $\alpha(\epsilon)\lesssim \epsilon$ then the operator $B_{\lambda d\sigma_\Gamma}$ inherits this boundedness. Indeed, it can be realized as the limit of the operators with symbol $\epsilon^{-1} (\lambda d\sigma_\Gamma) * \chi(\epsilon^{-1} \cdot)$, with $\chi$ a function in $\mathcal{C}^\infty_0$ with integral 1. In other words, $\Gamma$ satisfies then a $(p,q,r)$ bilinear restriction-extension inequality. 

Thus we deduce from theorems~\ref{thnc} and~\ref{thnvc} that if $\Gamma$ is non-characteristic, then $B_{d\sigma_\Gamma}$ is bounded for $p,q,r \geq 2$ and $\frac{1}{p} + \frac{1}{q} + \frac{1}{r} \geq 2$; whereas if the curvature of $\Gamma$ does not vanish, the condition becomes $\frac{1}{p} + \frac{1}{q} + \frac{1}{r} > \frac{5}{2}$.

This last condition is improved in Section~\ref{sec:bilinearrest}. Moreover, still in Section~\ref{sec:bilinearrest}, the relation between a bilinear restriction-extension inequality for exponents $(p,q,r)$ and the property (\ref{ecureuil}) for a decay rate $\alpha(\epsilon)=\epsilon$ is investigated. 
We shall prove that these are equivalent for specific symbols $m_\epsilon$ of the class ${\mathcal N}^\Gamma_\epsilon$.

\subsubsection{Bilinear Bochner-Riesz problem}

In Section~\ref{subsec:bochnerriesz}, conditions on a compact set $K$, and a real number $\lambda>0$, are deduced such that the operators with symbol
$$
\chi_K(\xi,\eta) \;\;\;\;\mbox{or}\;\;\;\; \chi_K(\xi,\eta) \operatorname{dist}((\xi,\eta),\partial K)^\kappa
$$
are bounded from $L^p \times L^q$ to $L^{r'}$.

As usual, this has the consequence, if $\textrm{Int}(K)$ contains $0$, and if we denote $\lambda K$ for the dilation of $K$ around 0 by a factor $\lambda$, that
$$
B_{\lambda K}^\kappa (f,g) \rightarrow fg \;\;\;\;\mbox{in $L^{r'}$}
$$
if $(f,g) \in L^p \times L^q$ and $\frac{1}{p}+\frac{1}{q} = \frac{1}{r'}$.

\subsubsection{Singular symbols}

In section~\ref{singularsymbols}, results on boundedness over Lebesgue spaces of operators with symbols of the type
$$
\Phi(\xi,\eta) \operatorname{dist}((\xi,\eta),\Gamma)^{-\alpha}
$$
(where $\Phi \in \mathcal{C}^\infty_0$, $\Gamma$ a smooth curve, and $\alpha$ a positive number) are given.

Recall that Kenig and Stein~\cite{KS} derived sharp results in the same spirit in the case when $\Gamma$ is a line; in this specific situation, the bilinear multiplier $B_{m_\epsilon}$ can be represented by bilinear fractional integral operators.  These authors were able to deal with Lebesgue exponents less than 1.

\subsubsection{Dispersive PDEs} \label{subsec:PDEs}

This paragraph is devoted to explaining one of the motivations for the study of our multilinear multipliers.

Let us consider the following dispersive PDE, with a real symbol $p$~:
$$ 
\left\{ \begin{array}{l}
\partial_t u + ip(D) u = B_m(u,u) \vsp \\
u(t=0)=u_0. \vsp 
\end{array} \right.
$$
where we follow the above notation in denoting $B_m$ for the pseudo-product with a symbol $m$; of course other nonlinearities can be dealt with in a similar way. For instance, the nonlinearity of the water-waves problem can be expanded as a sum of pseudo-product operators: see~\cite{GMS2}. Or it is well-known that a nonlinearity $H(u)$ can for many purposes be replaced by its paralinearization $ H(u(t,.)) \simeq \pi_{H'(u(t,.))} (u(t,.))$: see the seminal work of Bony~\cite{bony}).

In order to understand how $u$ behaves for large $t$, in particular whether it scatters, let us change the unknown function from $u$ to 
$$ f(t,x) = e^{itp(D)}[u(t,.)](x),$$
so that the PDE becomes
$$ \partial_t f = e^{itp(D)}[B_m(u,u)] = e^{itp(D)}[B_m(e^{-itp(D)}f, e^{-itp(D)}f)]$$
and integrating in time gives
\begin{align}
f(t,x) & = u_0(x) + \int_0^t e^{is p(D)} B_m(u(s),u(s))(x) \,ds \nonumber \\
& = u_0(x) + \int_0^t \int \int e^{ix\cdot(\xi+\eta)} e^{i s [p(\xi+\eta) - p(\eta) - p(\xi)]} \widehat{f}(s,\xi) \widehat{f}(s,\eta) m(\xi,\eta)\,d\xi\,d\eta\,ds   \label{duhamel}
\end{align}
we isolate in the right-hand side the oscillations in the term $e^{i[(s-t)p(\xi+\eta) - s p(\eta) - s p(\xi)]}$. With the phase function given by $\phi(\xi,\eta):=p(\xi+\eta) -  p(\eta) -  p(\xi)$, we get
\begin{equation} \label{eq:ff} f(t,x):=u_0(x) + e^{-itp(D)}\left[x\rightarrow \int_0^t \int \int e^{ix\cdot(\xi+\eta)} e^{is\phi(\xi,\eta)} \widehat{f}(s,\xi) \widehat{f}(s,\eta) m(\xi,\eta)\,d\xi\,d\eta\,d \right](x)s .
\end{equation}
So
$$ f(t)=u_0+ e^{-itp(D)} \left[\int_0^t  T_{e^{is\phi}m}(f(s),f(s)) ds\right].$$
Understanding the bilinear quantity above (for $f\in L^\infty_TL^p$) is of crucial importance. Let us consider first that $f$ is independent of $s$, which corresponds to examining the interaction of two linear waves. But the bilinear operator
\begin{equation} \label{eq:fff}
(f,g) \rightarrow  \int_0^t  B_{e^{is\phi}m}(f,g) ds
\end{equation}
has symbol $\displaystyle \frac{e^{it\phi} - 1}{i\phi}$ which, under suitable assumptions on $\phi$, fits into the previous setting, relatively to the curve $\Gamma:=\phi^{-1}(\{0\})$. We refer the reader to \cite{BG} (Section 9.3) for a first work of the authors concerning such bilinear oscillatory integrals. There, the proofs are based on decay in $s$ due to the assumed non-stationary phase $\phi$ and the set $\Gamma=\phi^{-1}(\{0\})$ did not play any role. Here we propose a more precise study in the one-dimensional case and when the symbol is supported near the curve $\Gamma$. We shall describe boundedness of bilinear quantities appearing in (\ref{eq:ff}) and (\ref{eq:fff}), see Subsection \ref{subsec:PDEs}.

\section{Notations and Preliminaries}
\label{notationsandpreliminaries}

\subsection{Some standard notations}

We adopt the following notations
\begin{itemize}
\item $A \lesssim B$ if $A \leq C B$ for some implicit, universal constant $C$. The value of $C$ may change from line to line.
\item $A \sim B$ means that both $A \lesssim B$ and $B \lesssim A$.
\item If $E$ is a set, $\chi_E$ is its characteristic function.
\item The ``japanese brackets`` $\langle \cdot \rangle$ stand for $\langle x \rangle = \sqrt{1+x^2}$.
\item If $\Gamma$ is a rectifiable curve, $d\sigma_\Gamma$ is the 1-dimensional Hausdorff measure restricted to $\Gamma$ and for $\epsilon>0$, we set $\Gamma_\epsilon$ for its $\epsilon$-neighborhood
$$ \Gamma_\epsilon := \left\{ \xi\in\R^2,\ d(x,\Gamma) \leq \epsilon\right\}.$$
\item ${\mathcal H}^1$ for the one-dimensional Hausdorff measure.
\item The standard $L^2$ scalar product is denoted $\langle f\,,\,g \rangle := \int_{\mathbb{R}^d}f \bar g$.
\item If $f$ is a function over $\mathbb{R}^d$ then its Fourier transform, denoted $\widehat{f}$, or $\mathcal{F}$, is given by
$$
\widehat{f}(\xi) = \mathcal{F}f (\xi) = \frac{1}{(2\pi)^{d/2}} \int e^{-ix\xi} f(x) \,dx \;\;\;\;\mbox{thus} \;\;\;\;f(x) = \frac{1}{(2\pi)^{d/2}} \int e^{ix\xi} \widehat{f}(\xi) \,d\xi.
$$
In the text, we systematically drop the constants such as $\frac{1}{(2 \pi)^{d/2}}$ since they are not relevant.
\item The Fourier multiplier with symbol $m(\xi)$ is defined by
$$
m(D)f = \mathcal{F}^{-1} \left[m \mathcal{F} f \right].
$$
\end{itemize}

\subsection{Some harmonic analysis}

Let us now recall some useful and well-known operators on $\R$.

\begin{df} \label{def:I}
For $s>0$, the fractional integral operator of order $s>0$ is defined by
$$ I_s(f)(x) := x \rightarrow \int_\R f(y) |x-y|^{s-1} dy.$$
For $s=0$, $I_0$ is not defined since $1/|\cdot|$ is not locally integrable, so we consider the right substitute: the Hardy-Littlewood maximal function:
$$ {\mathcal M}(f)(x):= \sup_{r>0} \frac{1}{2r} \int_{B(x,r)} |f(z)| dz.$$
For $s\geq 1$, we set $\M_s$ for its $L^s$-version defined by $\M_s(f) = [\M(|f|^s)]^{1/s}$.
\end{df}

These operators are bounded over in Lebesgue spaces.

\begin{prop}  \label{prop:fracint} Let $s\in (0,1)$ and let $1< p<q<\infty$ satisfying
 $$ \frac{1}{p} - \frac{1}{q} = s.$$
Then $I_s$ is bounded from $L^p$ to $L^q$ (see \cite{HL}).
For all $s\geq 1$ and $p\in (s,\infty]$, $\M_s$ is $L^p$-bounded.
\end{prop}

We recall the boundedness of Rubio de Francia's square functions (see \cite{RF}):

\begin{prop} \label{prop:square} Let $2\leq p < \infty$ and $I:=(I_i)_i$ be a bounded covering of $\R$. Then the square function 
$$ f \rightarrow \left(\sum_i \left| \pi_{I_i}(f) \right|^2 \right)^{1/2}$$
is $L^p$-bounded. The non-smooth truncations $\pi_{I_i}$ can be replaced by smooth ones.
\end{prop}

\subsection{Bilinear multipliers}

Let us recall some usual facts about bilinear multipliers. To a symbol $m\in \s(\R^2)'$, we can define in the distributional sense the following bilinear multiplier~:
$$ T_m(f,g)(x):= \int_{\R^2} e^{ix(\xi+\eta)} \widehat{f}(\xi) \widehat{g}(\eta) m(\xi,\eta) d\xi d\eta.$$
It is well-known that reciprocally, any translation invariant bilinear multiplier which is bounded from $\s(\R) \times \s(\R)$ into $\s(\R)'$ can be written in the previous form.

Seeing things in physical space, the bilinear operator $T_m$ can be represented as
\begin{equation}
\label{physical}
B_{m_\epsilon}(f,g)(x) = \sqrt{2 \pi} \int \int \widehat{m_\epsilon}(y-x,z-x) f(y) g(z)\,dy\,dz.
\end{equation}

\section{The setting of rough curves and extension of the results}
\label{sec:extensions}

In this subsection, we define some notions of rough curve, allowing us to extend the results. We first recall the notion of {\it rectifiable} and {\it Ahlfors regular} curve, which will then support a measure. Then, we precise some of our results which still hold in this general setting.

\begin{df}[Rectifiable curve] A \emph{rectifiable curve} is the image of a continuous map $\gamma:[0,1]\to\R^d$ with
$$
    \textrm{length}(\gamma) := \sup_{0\leq t_0\leq\cdots\leq t_n\leq 1} \sum_{j=1}^n |\gamma(t_j)-\gamma(t_{j-1})| < \infty.
$$
We refer the reader to \cite{DS,Falconer} for more details concerning rectifiable sets. If $\Gamma$ is compact then it is a rectifiable curve if and only if it is connected and  has a finite one-dimensional Hausdorff measure. Then, it is well-known that we can construct a finite measure $d\sigma_\Gamma$ (corresponding to the measure of the length) on $\Gamma$ such that $d\sigma_\Gamma$ is equivalent to ${\mathcal H}^1_{|\Gamma}$ (the restriction of the one-dimensional Hausdorff measure to the set $\Gamma$). For any measurable subset $E\subset \Gamma$
$$ \sigma_\Gamma(E) \simeq {\mathcal H}^1_{|\Gamma}(E) = \lim_{\delta\to 0} \inf_{\genfrac{}{}{0pt}{}{E\subset \cup_i U_i}{\textrm{diam}(U_i)\leq \delta}} \sum_{i} \textrm{diam} (U_i)$$
where $U_i$ are open sets.
\end{df}

\begin{df}[Ahlfors regular curve] \label{def:alfors} A continuous curve $\Gamma \subset \R^2$ is said {\it Ahlfors regular} if
\be{eq:density} \sup_{z \in \Gamma} \ \sup_{\epsilon\in (0,1)} \ \frac{{\mathcal H}^1(\Gamma \cap \overline{B}(z,\epsilon))}{\epsilon} <\infty. \ee
For $\epsilon\leq \textrm{length}(\Gamma)$ and $z\in \Gamma$, then there exists a part of $\Gamma$ of length larger than $\epsilon$ included into $\overline{B}(z,\epsilon)$, so we obviously have
\be{eq:densitybis} 1\leq \inf_{z \in \Gamma} \ \sup_{\epsilon<\textrm{length}(\Gamma)} \ \frac{{\mathcal H}^1(\Gamma \cap \overline{B}(z,\epsilon))}{\epsilon} <\infty. \ee
\end{df}

\mb We refer the reader to \cite{DS} for the analysis of and on such curves, satisfying these regularity.

\begin{prop} \label{prop:weak} Let $\Gamma$ be a continuous Ahlfors regular curve. Then there exists a sequence of symbols $m_{\epsilon} \in \mathcal{M}_\epsilon^\Gamma$ for $\epsilon>0$ such that
$ \epsilon^{-1} m_\epsilon$ weakly converges to $d\sigma_\Gamma$ in the distributional sense.
\end{prop}

\begin{proof} First since $\Gamma$ satisfies (\ref{eq:density}), it is obvious that $\Gamma$ is a rectifiable curve and so admits a arc-length measure (denoted $d\sigma_\Gamma$). Moreover it is easy to see that (\ref{eq:density}) self-improves in
\be{eq:density2} \sup_{z \in \R^2} \ \sup_{\epsilon\in (0,1/2)} \ \frac{{\mathcal H}^1(\Gamma \cap \overline{B}(z,\epsilon))}{\epsilon} <\infty \ee
Let choose a nonnegative smooth function $\chi:\R^2 \to \R$ supported on $\overline{B}(0,1)$ with $\int_{\R^2} \chi(x) dx=1$. Then for $\epsilon\in(0,1/2)$, we consider the following symbols
$$ m_\epsilon(x):= \int_\Gamma \frac{1}{\epsilon} \chi\left(\frac{x-y}{\epsilon}\right) d\sigma_\Gamma(y).$$
It is easy to see that $m_\epsilon$ is supported on $\Gamma_\epsilon$ and for all $\alpha\in\N^2$
\begin{align*}
 \left\| \partial^\alpha m_\epsilon\right\|_\infty & \lesssim \epsilon^{-|\alpha|} \sup_{x\in \Gamma_\epsilon} \frac{\sigma_\Gamma(\overline{B}(x,\epsilon))}{\epsilon}  \\
 & \lesssim \epsilon^{-|\alpha|} \sup_{x\in \Gamma_\epsilon} \frac{{\mathcal H}^1(\Gamma\cap \overline{B}(x,\epsilon))}{\epsilon} \lesssim \epsilon^{-|\alpha|},
\end{align*}
where we have used (\ref{eq:density2}). Consequently, the symbol $m_\epsilon$ belong to $\mathcal{M}_\epsilon^\Gamma$. In addition, for a compactly smooth function $f\in C^\infty_0(\R^2)$, we have 
\begin{align*}
 \epsilon^{-1} \int_{\R^2} m_\epsilon(x) f(x) dx & = \int_{\Gamma \times \R^2} f(x) \chi\left(\frac{x-y}{\epsilon}\right) \frac{dx d\sigma_\Gamma(y)}{\epsilon^2} \\
& = \int_{\Gamma \times \R^2} f(y+\epsilon z) \chi(z) dx d\sigma_\Gamma(y)dz.
\end{align*}
Hence, since $f$ is uniformly continuous and compactly supported, we easily check that
$$ \lim_{\epsilon \to 0} \epsilon^{-1} \int_{\R^2} m_\epsilon(x) f(x) dx = \int_{\Gamma} f(y) d\sigma_\Gamma(y),$$
thanks to the $L^1$-normalization of $\chi$. That concludes the proof of the proposition.
\end{proof}

\begin{lem} \label{lem:gammaeps} If $\Gamma$ is a rectifiable curve of finite length then there exists a constant $c_\Gamma$ such that for small enough $\epsilon>0$
$$ |\Gamma_\epsilon| \leq c_\Gamma \epsilon.$$
\end{lem}

\begin{proof} Consider $\epsilon$ small enough with respect to the length of $\Gamma$ and choose a nonnegative smooth function $\chi:\R^2 \to \R$ supported on $\overline{B}(0,1)$ with $\chi(x)=1$ for $|x|\leq 1$.
 For all $z\in \Gamma_\epsilon$,  we have
 $$ \frac{1}{\epsilon} \int_\Gamma \chi\left( \frac{z-y}{2\epsilon} \right) d\sigma_\Gamma(y) \geq \frac{{\mathcal H}^1(\Gamma \cap \overline{B}(z,2\epsilon))}{\epsilon}.$$
 However since $z\in \Gamma_\epsilon$, there exists $x\in \Gamma \cap \overline{B}(z,\epsilon)$, hence
 $$ {\mathcal H}^1(\Gamma \cap \overline{B}(z,2\epsilon)) \geq {\mathcal H}^1(\Gamma \cap \overline{B}(x,\epsilon)).$$
 Since $\epsilon$ is assumed to be smaller than the length of $\Gamma$, it follows that there exists at most a part of $\Gamma$ of length larger than $\epsilon$ included into $\overline{B}(x,\epsilon)$. We also deduce that
 $$ \frac{1}{\epsilon} \int_\Gamma \chi\left( \frac{z-y}{2\epsilon} \right) d\sigma_\Gamma(y) \geq 1.$$
 Consequently, 
\begin{align*} 
|\Gamma_\epsilon| & \leq \int_{\R^2} \frac{1}{\epsilon} \int_\Gamma \chi\left( \frac{z-y}{2\epsilon} \right) d\sigma_\Gamma(y) dz \\
 & \lesssim  \epsilon \left( \int_{\R^2} \chi( x)dx \right)  \left( \int_\Gamma d\sigma_\Gamma(y) \right) \\
 & \lesssim \epsilon ,
 \end{align*}
 where we have used that $\Gamma$ has a finite length.
 \end{proof}

Then, to extend the results, we have to perform a suitable decomposition of the curve. We also introduce the two following notions~:

\begin{df} \label{def:biLip} Let $\Gamma$ be a Ahlfors-regular curve in $\R^2$. Then consider $\pi_1,\pi_2,\pi_3$, respectively the orthogonal projection on the degenerate line $\{\eta=0\}$, $\{\xi=0\}$ and $\{\xi+\eta=0\}$. The curve $\Gamma$ is said to have ``finitely bi-Lipschitz projections'' if $\Gamma$ can be split into several pieces $(\Gamma_i)_{i=1...N}$ with for every $i$, a bi-Lipschitz parametrization of $ \pi_k(\Gamma_i)$ for at least two indices $k\in\{1,2,3\}$.
\end{df}

We know that an Ahlfors-regular curve can be split into bi-Lipschitz parametrized pieces. Here we required that these pieces have bi-Lipschitz parametrization through some projections.

Obviously, smooth curves satisfy this property, as well as polygons. We point out that some Ahlfors-regular curves (even with a finite length) may not respect this property, for example consider a logarithmic spiral.

\begin{df} A Ahlfors-regular curve $\Gamma$ in $\R^2$ is said to be nowhere characteristic if there exists a constant $c$ such that
for all real $t\in\R$, for every characteristic angle $\theta \in\{0,\pi/2,3\pi/4\}$ and all small enough $\epsilon>0$
\be{ass:nondegenerate} {\mathcal H}^1( \Gamma \cap \{(\xi,\eta),\ t\leq \xi-\tan(\theta)\eta \leq t+\epsilon\}) \leq c \epsilon. \ee
This assumption describes that the curve $\Gamma$ has no tangential directions (Ahlfors regular curve admits almost every where a tangential vector) which would be characteristic. \\
It is easily to check that a nowhere characteristic curve has ``finitely bi-Lipschitz projections".
\end{df}

With these two notions, we let the reader check the following possible extensions (we only give a sample of these).

\begin{prop} \label{prop:extension} We can extend the results as follows~:
\begin{itemize}
\item Let $\Gamma$ be a Ahlfors regular curve in $\R^2$ having ``finitely bi-Lipschitz projections'' (not necessarily bounded) then for $p,q,r\in[2,\infty)$ there exists a constant $c$ such that for all symbols $m_\epsilon \in \mathcal{M}^\Gamma_\epsilon$,
$$\displaystyle \|B_{m_\epsilon}\|_{L^p \times L^q \rightarrow L^{r'}} \leq c \epsilon^{\frac{1}{p}+\frac{1}{q}+\frac{1}{r}-1}.$$
\item If $\Gamma$ is nowhere characteristic then we can allow one of the three exponents to be lower than $2$.
\item Let $\Gamma$ be Ahlfors regular, nowhere characteristic, then bilinear Fourier restriction-extension inequalities still holds for $\frac{1}{p}+\frac{1}{q}+\frac{1}{r}>2$ (since (\ref{ecureuil}) is true with $\alpha(\epsilon)\lesssim \epsilon$).
\end{itemize}
\end{prop}

Indeed, the technical point relies on this following lemma, which gives a suitable decomposition at the scale $\epsilon$~:

\begin{lem} \label{lem:covering} 
Let $\Gamma$ be a Ahlfors regular curve in $\R^2$ having ``finitely bi-Lipschitz projections'' (see above Definition \ref{def:biLip}) and fix $\epsilon>0$. Then there exist a number $N$ (independent on $\epsilon>0$) and $N$ collections of balls ${\mathcal B}_j:=(B(z_i,\epsilon))_{i\in I_j}$ with $z_i\in \Gamma$ satisfying the following property: \\
Let us write $I_{j,i}^1$, $I_{j,i}^2$ and $I_{j,i}^3$ intervals of length equal to $2\epsilon$ such that for all $i\in I_j$
\be{eq:inter} \left\{(\xi,\eta,\xi+\eta), (\xi,\eta)\in \overline{B}(z_i,\epsilon)\right\} \subset I_{j,i}^1 \times I_{j,i}^2 \times I_{j,i}^3. \ee
For $l=1,2$ and $l=3$, there exists a subset $ I_j^l \subset I_j$ such that $I_j = I_j^1\cup I_j^2 \cup I_j^3$ and for each $l\in \{1,2,3\}$ and $k\in\{1,2,3\}\setminus \{l\}$, the intervals $(I_{j,i}^k)_{i\in I_j^l}$ are disjoint.
Here the constant $N$ is a numerical constant, independent with respect to $\epsilon$.
\end{lem}

We let the details to the reader. The proof is a direct consequence of the geometrical assumptions "finitely bi-Lipschitz projections". Indeed if the curve $\Gamma$ is only assumed to be bi-Lipschitz, then by considering $\gamma$ a bi-Lipschitz parametrization, we have that $|x-y|\lesssim |\gamma(x) - \gamma(y)|$ for all $x,y$. It follows that for at least two integers $k,k'$, $|x-y|\lesssim |\pi_k(\gamma(x))-\pi_k(\gamma(y))|$ (we recall that $\pi_k$ is the projection on one of the degenerate lines). The main difficulty is that these two integers $k,k'$ may depend on the couple $(x,y)$. The geometrical assumption allows us to deal with same integers $k,k'$ for all points $x,y$.

\section{Necessary conditions and the specific case of a straight line for $\Gamma$}

\label{NCATSC}

\subsection{Necessary conditions}

It is well-known that multilinear multipliers (commuting with the simultaneous translations) cannot loose integrability. So any bilinear multiplier can be bounded from $L^p \times L^q$ into $L^{r'}$ only if
$$ \frac{1}{r'} \leq \frac{1}{p}+ \frac{1}{q}$$ which corresponds to (\ref{eq:subholder}).

\begin{rem}
\label{jaguar}
Next notice that, since $\Gamma$ is compactly supported, (\ref{ecureuil}) holds for $(p,q,r,\alpha(\epsilon))$ if it does hold for $(P,Q,R,\alpha(\epsilon))$ with $p \leq P$, $q \leq Q$ and $r \leq R$. Thus our job will be to push indices up !
\end{rem}

Let $\Gamma$ be a smooth and compact curve in $\R^2$. Let us recall that we are looking for exponents $p,q,r\in [1,\infty]$ verifying (\ref{eq:subholder}) and function $\alpha(\epsilon)$ such that for small enough $\epsilon>0$, any symbol $m_\epsilon$ in the class $\mathcal{M}_\epsilon^\Gamma$ or $\mathcal{N}_\epsilon^\Gamma$ gives rise to a bilinear multiplier $B_{m_\epsilon}$ with
\be{ecureuil2} \left\|B_{m_\epsilon}\right\|_{L^p \times L^q \rightarrow L^{r'}} \lesssim \alpha(\epsilon).\ee
This subsection is devoted to describe some necessary conditions.

\begin{prop}[Necessary conditions]\label{fourmi}
Assume that $m_\epsilon \in \mathcal{N}_\epsilon^\Gamma$, $m_\epsilon \geq 0$, and $m_\epsilon = 1$ on a (non empty) curve $\Gamma' \subset \Gamma$. With $p,q,r$ verifying (\ref{eq:subholder}), for~(\ref{ecureuil2}) to hold, it is necessary that
\begin{itemize}
\item[(i)] $\alpha(\epsilon) \gtrsim \epsilon$ (for any $\Gamma$), \vsp
\item[(ii)] $\alpha(\epsilon) \gtrsim \epsilon^{\frac{1}{p} + \frac{1}{q} + \frac{1}{r} - 1}$ (for any $\Gamma$), \vsp
\item[(iii)] $\alpha(\epsilon) \gtrsim \epsilon^{-\frac{1}{2}+\frac{1}{q}+\frac{1}{2}\left(\frac{1}{p}+\frac{1}{r}\right)+\delta }$ for $\delta>0$ (if $\Gamma'$ has non -vanishing curvature and has, in the coordinates $(\xi,\eta)$, a tangent which is parallel to the $\xi$ axis).
\item[(iv)] $\alpha(\epsilon) \gtrsim \epsilon^{\frac{1}{p}}$ (if $m_\epsilon(\xi,\eta) = \chi\left(\frac{\xi}{\epsilon}\right) \chi(\eta)$, with $\chi$ a smooth function, equal to $1$ on $B(0,1)$, and $0$ outside of $B(0,2)$).
\end{itemize}
\end{prop}

\begin{proof} Let consider a suitable nonnegative symbol $m_\epsilon$.
$(i)$ Take $R$ such that $\operatorname{Supp} m_\epsilon \subset B(0,R)$, and $f$, $g$, and $h$ in $\mathcal{S}(\R)$ such that $\widehat{f} = \widehat{g} = \widehat{h} = 1$ on $B(0,10 R)$. Then an obvious computation with Lemma \ref{lem:gammaeps} give
$$
\langle B_{m_\epsilon}(f,g) \,,\,h \rangle \sim \epsilon,
$$
hence the bound on $\alpha$.
\bigskip

$(ii)$ Assume that $\Gamma'$ goes through $(0,0)$, and take $f$, $g$, $h$ in $\mathcal{S}(\R)$ such that $\widehat{f}, \widehat{g}, \widehat{h} \geq 0$ and $\operatorname{Supp} \widehat{f}$, $\operatorname{Supp} \widehat{g}$, $\operatorname{Supp} \widehat{h} \subset B(0,\frac{1}{100})$. Define then
$$
f^\epsilon = \epsilon f(\epsilon \cdot), \qquad g^\epsilon = \epsilon g(\epsilon \cdot)\quad \mbox{and}\quad h^\epsilon = \epsilon h(\epsilon \cdot).
$$
Then obviously
\begin{align*}
\langle B_{m_\epsilon}(f^\epsilon ,g^\epsilon ) \,,\,h^\epsilon \rangle & = \epsilon^ 2 \int\int \widehat{f}(\xi) \widehat{g}(\eta) \widehat{h}(-\xi-\eta) m_\epsilon(\epsilon\xi,\epsilon \eta) d\xi d\eta \\
 & \sim \epsilon^2
\end{align*}
whereas
$$
\left\| f^\epsilon \right\|_{L^p} \left\| g^\epsilon \right\|_{L^q} \left\| h^\epsilon \right\|_{L^r} \sim \epsilon^{3 - \frac{1}{p} - \frac{1}{q} - \frac{1}{r}};
$$
this gives the bound on $\alpha$.

\bigskip

$(iii)$ We only treat the case where $p,q,r > 1$, a small modification being needed if one of the exponents is $1$. With the hypotheses made on $\Gamma'$, this curve can be parameterized in some region by $\eta=\phi(\xi)$, with $\phi'$ vanishing at a point, $\phi'(\xi_0)=0$, but $\phi''(\xi_0)\neq 0$ since the curvature of $\Gamma'$ does not vanish. For simplicity, we shall assume that $\phi'(0)=0$, and $\phi''(0)=1$. The test functions we will use read
$$
\widehat{f}(\xi) = \frac{\Phi(\xi)}{|\xi|^a}\;\;,\;\;\widehat{g}(\xi) = \frac{\Phi(\xi)}{|\xi|^b}\;\;\mbox{and}\;\;\widehat{h}(\xi)=\frac{\Phi(\xi)}{|\xi|^c},
$$
where $\Phi \in \mathcal{C}^\infty_0$ is non negative and does not vanish at $0$, and
$$
a = 1-\frac{1}{p}-\frac{\delta}{3}, \quad b = 1-\frac{1}{q}-\frac{\delta}{3}\quad \mbox{and}\quad c = 1-\frac{1}{q}-\frac{\delta}{3}
$$
for some small $\delta>0$. It is easy to check that $f$, $g$, and $h$ belong, respectively, to $L^p$, $L^q$, and $L^r$. Next we want to estimate $\langle B_{m_\epsilon}(f,g) \,,\,h \rangle$. For some appropriate constant $c_0$,
$$
\langle B_{m_\epsilon}(f,g) \,,\,h \rangle = \int \int m_\epsilon(\xi,\eta) \widehat{f}(\xi) \widehat{g}(\eta) \widehat{h}(-\eta-\xi)\,d\eta\,d\xi \gtrsim \int_{|\eta| \leq c_0 \epsilon} \int_{|\xi|\sim c_0 \sqrt{\epsilon}} \frac{1}{|\xi|^a} \frac{1}{|\eta|^b} \frac{1}{|\xi+\eta|^c} \,d\eta \, d\xi.
$$
An easy computation gives then
$$
\langle B_{m_\epsilon}(f,g) \,,\,h \rangle \geq \epsilon^{-\frac{1}{2}+\frac{1}{q}+\frac{1}{2}\left(\frac{1}{p}+\frac{1}{r}\right)+3\delta}
$$
which is the desired result.

\bigskip

$(iv)$ With $m_\epsilon = \chi\left(\frac{\xi}{\epsilon}\right) \chi(\eta)$, $B_{m_\epsilon}(f,g) = \chi\left(\frac{D}{\epsilon}\right) f \chi(D) g$. Choosing $f,g,h$ in the Schwartz class with Fourier transforms localized in $B \left( 0,\frac{1}{2} \right)$, set furthermore $f^\epsilon = f(\epsilon \cdot)$.
Then
$$
< B_{m_\epsilon}(f,g)\,,\,h> = \int f(\epsilon \cdot) g h \rightarrow f(0) \int gh \;\;\;\;\;\;\mbox{as $\epsilon \rightarrow 0$}
$$
whereas
$$
\left\|f^\epsilon\right\|_{L^p} \|g\|_{L^q} \|h\|_{L^r} = \epsilon^{-\frac{1}{p}} \|f\|_{L^p} \|g\|_{L^q} \|h\|_{L^r}.
$$
This implies immediately $(iv)$.
\end{proof}

\subsection{Necessary and sufficient conditions in the case $m_\epsilon = \left[ \lambda d\sigma_{\Gamma} \right] * \epsilon^{-1} \chi_\epsilon$}

We assume in this subsection that 
$$
m_\epsilon = \left[ \lambda d\sigma_{\Gamma} \right] * \epsilon^{-1} \chi_\epsilon,
$$
where $\chi_\epsilon = \chi(\epsilon^{-1} \cdot)$ a smooth, nonnegative function, supported on $B(0,2)$, equal to one on $B(0,1)$; $\lambda$ belongs to $\mathcal{C}^\infty_0$; and $d\sigma_\Gamma$ is the arc-length measure on the smooth curve $\Gamma$. Notice that this type of $m_\epsilon$ belongs to $\mathcal{M}_\epsilon$; and it belongs to $\mathcal{N}_\epsilon^\Gamma$  if $\lambda = 1$, $\Gamma = \mathbb{S}^1$.

\begin{prop} 
\begin{itemize}
\label{macareux}
\item[(i)] Let $m_\epsilon = d\sigma_{\Gamma} * \epsilon^{-1} \chi_\epsilon$. Then if the curvature of $\Gamma$ does not vanish,
$$
\left\| B_{m_\epsilon} \right\|_{L^{1}\times L^1 \rightarrow L^{p'}} \sim \left\{ \begin{array}{ll} \epsilon & \mbox{if $2 < p'<\infty$} \vsp \\ \epsilon \sqrt{- \log \epsilon}  & \mbox{if $p'=p=2$} \vsp \\ \epsilon^{1/2+1/p} & \mbox{if $1 < p' < 2$} \vsp  \end{array} \right.
$$
\item[(ii)] We have the restricted type estimate at the point $(1,1,2)$: for any three sets $F$, $G$, and $H$,
$$
\left| \left< B_{m_\epsilon} (\chi_F,\chi_G)\,,\,\chi_H \right> \right| \lesssim \epsilon |F| |G| |H|^{1/2}.
$$
\item[(iii)] If $\Gamma$ is non-characteristic on $\operatorname{Supp} \lambda$, the only change to $(i)$ is the point $(1,1,2)$ which becomes
$$
\left\| B_{m_\epsilon} \right\|_{L^1 \times L^1 \rightarrow L^2} \sim \epsilon.
$$
\end{itemize}
\end{prop}

The curvature of the curve $\Gamma$ can be used via the following result (we move the reader to Section 3.1 Chap VIII of \cite{Stein} for details on this topic), the proofs rely on the study of oscillatory integrals.

\begin{lem} \label{lem:kernelosci} Assume that $\Gamma$ is a smooth curve in $\R^2$ with a non-vanishing Gaussian curvature. Then for all compactly smooth function $\psi$ in $\R^2$, we have
$$ \left|\int_{\R^2} e^{i(x_1\xi+x_2\eta)} \psi(\xi,\eta) d\sigma_{\Gamma}(\xi,\eta) \right|\lesssim (1+|(x_1,x_2)|)^{-1/2},$$
where $d\sigma_\Gamma$ is the carried surface measure on $\Gamma$. 
\end{lem}

\begin{proof} \noindent \underline{Another expression for $\| B_{m_\epsilon} \|_{L^1 \times L^1 \rightarrow L^{p'}}$} Suppose $B$ is a general bilinear operator with a kernel $K(x,y,z)$ which is smooth and decaying at infinity:
$$
B(f,g)(x) = \int \int K(x,y,z) f(y) g(z) \,dy\,dz.
$$
Then
\begin{equation}
\label{pilgrimbis}
\|B\|_{L^1 \times L^1 \rightarrow L^{q}} = \sup_{y,z} \| K(\cdot,y,z) \|_{L^q}.
\end{equation}
The first side of the above equality is straightforward:
$$
\|B(f,g)(x)\|_{L^q} = \left\| \int \int K(x,y,z) f(y) g(z) \,dy\,dz \right\|_{L^q} \lesssim \sup_{y,z} \| K(x,y,z) \|_{L^q} \|f\|_{L^1} \|g\|_{L^1}.
$$
To see the other direction, let $(f_n)=\left(n^{-1}\phi\left(\frac{\cdot}{n}\right)\right)$, with $\phi \in \mathcal{C}^\infty_0$, be an approximation of the Dirac delta function: this sequence has constant norm 1 in $L^1$, but converges to $\delta$ weak-star in the sense of bounded measures. Then
$$
B(f_n(\cdot-y_0),f_n(\cdot-z_0))(x) \rightarrow K(x,y_0,z_0)\;\;\;\;\mbox{in $L^1$ as $n \rightarrow \infty$}
$$
This implies $\|B\|_{L^1 \times L^1 \rightarrow L^{q}} \geq \|K(\cdot,y_0,z_0) \|_{L^q}$, hence the equality~(\ref{pilgrimbis}). Translating~(\ref{pilgrimbis}) in the context of our problem, this gives
\begin{equation}
\label{milkywaybis}
\|B_{m_\epsilon}\|_{L^1 \times L^1 \rightarrow L^{p'}} = \sup_{y,z} \left\| \widehat{m_\epsilon}(y-x,z-x) \right\|_{L^{p'}}.
\end{equation}

\bigskip

\noindent \underline{Estimates on the convolution kernels} It is well-known (see Lemma~\ref{lem:kernelosci}), that, if the curvature of $\Gamma$ does not vanish, we have the bound
$$
\left| \widehat{\lambda d\sigma_{\Gamma}} (x) \right| \lesssim \frac{1}{\langle x \rangle^{1/2}} \;\;\;\mbox{for $x$ in $\mathbb{R}^2$}.
$$
Furthermore, examining further the stationary phase argument which gives the above, we find the following: suppose that the normals to $\Gamma$ on $\operatorname{Supp} \lambda$ span a subset $[\alpha,\beta]$ of $\mathbb{S}^1$. Then if $\frac{(x_1,x_2)}{|(x_1,x_2)|}$ belongs to $[\alpha+\epsilon,\beta-\epsilon]$, for a small $\epsilon$,
$$
\left| \widehat{\lambda d\sigma_{\Gamma}} (x) \right| \gtrsim \frac{1}{\langle x \rangle^{1/2}}.
$$
On the other hand, if $\frac{(x_1,x_2)}{|(x_1,x_2)|}$ does not belong $[\alpha-\epsilon,\beta+\epsilon]$, for a small $\epsilon$,
$$
\left| \widehat{\lambda d\sigma_{\Gamma}} (x) \right| \lesssim \frac{1}{\langle x \rangle^{N}}.
$$
Finally, recall that 
$$
\mathcal{F} \left[ \lambda d\sigma_{\Gamma} \right] * \epsilon^{-1} \chi_\epsilon = \mathcal{F} \left[ \lambda d\sigma_{\Gamma} \right] \epsilon \widehat{\chi} (\epsilon \cdot).
$$
The above bounds, combined with~(\ref{milkywaybis}), give the desired bounds, except for the restricted type estimates, to which we now turn.

\bigskip

\noindent  \underline{The restricted type estimate} For these estimates, we argue with the physical space version of $B_{m_\epsilon}$. Recall $m_\epsilon = \chi_\epsilon * d\sigma_{\Gamma}$. Thus
$$
\widehat{m_\epsilon} = \epsilon \widehat{\lambda d\sigma_{\Gamma}} \widehat{\chi_\epsilon}.
$$
Since $\chi$ is in the Schwartz class, so is $\widehat{\chi}$; since $\Gamma$ has a non-vanishing curvature, it is well known that
$$
\left| \widehat{\lambda  d\sigma_{\Gamma}}(X) \right| \lesssim \frac{1}{\sqrt{\langle X \rangle}}.
$$
Thus
$$
| \widehat{m_\epsilon}(X)| \lesssim \frac{1}{\sqrt{\langle X \rangle}}.
$$
This implies
$$
| \widehat{m_\epsilon}(x-y,x-z)| \lesssim \epsilon \frac{1}{\sqrt{\langle (x-y,x-z) \rangle}} \lesssim \epsilon \frac{1}{\sqrt{\langle (x-y,z-y) \rangle}}.
$$
Therefore,
\begin{equation}
\begin{split}
\left| \langle B_{m_\epsilon} (\chi_E,\chi_F)\,,\,\chi_G) \right| & \lesssim \epsilon \int \int \int 
\frac{1}{\sqrt{\langle (x-y,z-y) \rangle}} \chi_F(y) \chi_G(z) \chi_H(x)\,dx\,dy\,dz \\
& \lesssim \epsilon| F| \sup_y \int \int \frac{1}{\sqrt{\langle (x-y,z-y) \rangle}}
 \chi_G(z) \chi_H(x)\,dx\,dz
\end{split}
\end{equation}
By symmetry in $y$, the desired estimate will be a consequence of the inequality
$$
\int \int \frac{1}{\sqrt{\langle (x,z) \rangle}} \chi_G(z) \chi_H(x)\,dx\,dz \lesssim |G| |H|^{1/2}.
$$
But this inequality follows from
\begin{equation}
\begin{split}
\int \int \frac{1}{\sqrt{\langle (x,z) \rangle}} \chi_G(z) \chi_H(x)\,dx\,dz & \leq \int \chi_G(z) \int \frac{1}{\sqrt{\langle x \rangle}} \chi_H(x)\,dx\,dz\\
& \leq |G| \int_0^{|H|} \frac{1}{\sqrt{\langle x \rangle}}\,dx \lesssim |E| |F|^{1/2}.
\end{split}
\end{equation}
\end{proof}

\subsection{The specific case where $\Gamma$ is a line}

Consider a symbol of the type $m_\epsilon(\xi,\eta) = \chi \left( \frac{ \xi - \lambda \eta}{\epsilon} \right)$ with $\chi$ a smooth function supported in $\overline{B(0,1)}$. Such symbols belong to $\mathcal{M}_\epsilon^\Gamma$ relatively to the line
$$ \Gamma:=\{(\xi,\eta), \ \xi=\lambda \eta\}.$$
The degenerate lines are those corresponding for $\lambda\in\{0,-1\}$.

\subsubsection{The non-characteristic case}

In that case, we have
$$
B_{m_\epsilon}(f,g) (x) = \int_{\R} \epsilon \widehat{\chi} (\epsilon y) f (x+ y) g(x-\lambda y) \,dy
$$
thus
\be{eq:droite}
\langle B_{m_\epsilon}(f,g)\,,\,h \rangle = \int_{\R^2} \epsilon \widehat{\chi} (\epsilon y) f (x+y) g(x-\lambda y) h(x) \,dy dx.
\ee

\begin{prop} If $\lambda \neq 0,-1$ then
\begin{equation}
\label{moineau}
\left| \langle T_{m_\epsilon}(f,g)\,,\,h \rangle \right| \lesssim \epsilon^{\rho} \|f\|_{L^p} \|g\|_{L^q} \|h\|_{L^r}.
\end{equation}
holds if and only if
$$
1 \leq \rho + 1 = \frac{1}{p} + \frac{1}{q} + \frac{1}{r} \leq 2.
$$
\end{prop}

\begin{proof} Let us just point out some easy observations, whose proofs we leave to the reader.
\begin{enumerate}
\item First of all, the scaling imposes $\rho = \frac{1}{p}+\frac{1}{q}+\frac{1}{r}-1$.
\item We already have explained that necessarily $\frac{1}{p}+\frac{1}{q}+\frac{1}{r} \geq 1$.
\item Next, the exponent $\frac{1}{p}+\frac{1}{q}+\frac{1}{r}-1$ cannot be larger than 1; otherwise you get a vanishing limit as $\epsilon$ goes to zero for the ``restriction problem''.
\item It is easy to see (by Young's inequality) that~(\ref{moineau}) holds for $\frac{1}{p}+\frac{1}{q}+\frac{1}{r} = 1$ if $1 \leq p,q,r \leq \infty$ (with $\rho=0$).
\item Finally, the case $\frac{1}{p}+\frac{1}{q}+\frac{1}{r}=2$ with $1 \leq p,q,r \leq \infty$ and $\rho=1$ follows easily from (\ref{eq:droite}) by H\"older's inequality with respect to $x$ and then Young inequality.
\end{enumerate}
Interpolating between the two last points, and taking into account the necessary conditions derived above, ends the proof.  
\end{proof}

\subsubsection{The characteristic case}

Let us now consider one of the degenerate lines, for instance when $\lambda=0$. In that case, we have
$$
T_{m_\epsilon}(f,g) = \left[ \chi \left( \frac{D}{\epsilon} \right) f \right] g
$$
and so
$$
\langle T_{m_\epsilon}(f,g)\,,\,h \rangle = \int \left[ \chi \left( \frac{D}{\epsilon} \right) f \right] g h.
$$
\begin{prop} If $\lambda = 0$ then
\begin{equation}
\label{mesange}
\left| \langle T_{m_\epsilon}(f,g)\,,\,h \rangle \right| \lesssim \epsilon^{\rho} \|f\|_{L^p} \|g\|_{L^q} \|h\|_{L^r}.
\end{equation}
holds if and only if
$$
1 \leq \rho + 1 = \frac{1}{p} + \frac{1}{q} + \frac{1}{r} \leq 2 \quad \textrm{with} \quad \frac{1}{q} + \frac{1}{r} \leq 1.
$$
\end{prop}

We do not detail the proof. Indeed points 1 to 4 of the previous proof remain valid. Furthermore, it is easy to see that the following condition in necessary: $\frac{1}{q}+\frac{1}{r} \leq 1$. Finally, it is easy to see that the case $(p,q,r) = (1,q,r)$ with $\frac{1}{q}+\frac{1}{r} = 1$ is admissible, which allows us to conclude.

\section{The local-$L^2$ case with finite exponents} \label{sec:L2}

\mb Let us first study the case where the three exponents $p,q,r$ belong to $[2,\infty)$.

\begin{prop} \label{prop:L2} Consider $\Gamma$ a compact and smooth curve. Let $p,q,r\in [2,\infty)$ be exponents satisfying
$$ 2s:= \frac{1}{p}+\frac{1}{q}+\frac{1}{r} -1\geq 0.$$
Then there exists a constant $C=C(p,q,r)$ such that for every $\epsilon>0$ and symbols $m_\epsilon\in \mathcal{M}_\epsilon^\Gamma$ , then  
$$ \left\|T_{m_\epsilon}(f,g) \right\|_{L^{r'}} \leq C \epsilon^ {\frac{1}{p}+\frac{1}{q}+\frac{1}{r}-1} \|f\|_{L^p} \| g\|_{L^q}.$$
Moreover, Proposition \ref{fourmi} implies that the decay in $\epsilon$ is optimum.
\end{prop}

\begin{proof}  First, the domain $\Gamma_\epsilon$ can be covered by balls of radius $\epsilon$ with bounded intersection. A partition of the unity associated to this covering allows us to split the symbol $m_\epsilon$ as follows~:
$$ m_\epsilon = \sum_{i\in \Theta } m_\epsilon^i$$
where for each index $i$, $m_\epsilon^i$ is a symbol satisfies the same regularity as $m_\epsilon$ and is supported in a ball of radius $\epsilon$.
Let us write $I_i^1$, $I_i^2$ and $I_i^3$ for intervals of length comparable to $\epsilon$ such that
$$\left\{ (\xi,\eta,\xi+\eta), m_\epsilon^i (\xi,\eta) \neq 0\right\} \subset I_i^1 \times I_i^2 \times I_i^3.$$
For $J$ an interval, we write $\Delta_J$ the Fourier multiplier associated to the symbol $\phi_J$ (a smooth version of ${\bf 1}_J$ at the scale $|J|$ such that ${\bf 1}_J \leq \phi_J \leq {\bf 1}_{2J}$). \\
The kernel $\widehat{m_\epsilon}$ satisfies
$$ \left|\widehat{m_\epsilon^i}(y,z)\right|\lesssim \frac{\epsilon^2}{(1+\epsilon|(y,z)|)^N}$$
for all nonnegative real $N$ (since $m_\epsilon^i$ is supported on a ball of radius $\epsilon$). Hence, the operator $T_{m_\epsilon^i}$ satisfies
$$ \left|\langle T_{m_\epsilon^i}(f,g) , h\rangle \right| \lesssim \int_{\R^3} \frac{\epsilon^2}{(1+\epsilon|(x-y,x-z)|)^N} |\Delta_{I_i^1} f(y)| | \Delta_{I_i^2} g(z)| | \Delta_{I_i^3} h(x)| dxdydz.$$
Then
$$ \left| \langle T_m(f,g),h\rangle\right| \lesssim \sum_{i} \int_{\R^3} \frac{\epsilon^2}{(1+\epsilon|(x-y,x-z)|)^N} |\Delta_{I_i^1} f(y)| | \Delta_{I_i^2} g(z)| | \Delta_{I_i^3} h(x)| dxdydz.$$
By Lemma~\ref{lem:covering}, it suffices to treat the case where $(I_i^1)$ and $(I_i^2)$ form a bounded covering of the real line.
Let assume that $s>0$ (we explain at the end of the proof the modifications for $s=0$). Consider non-negative real numbers $s_p,s_q\in (0,1)$ such that $2s=s_p+s_q$ and
$$  \frac{1}{p}-s_p >0 \qquad \frac{1}{q}-s_q >0.$$
This is possible since
$$ \frac{1}{r} = \frac{1}{p}+\frac{1}{q} -2s >0.$$
Using that 
$$ (1+\epsilon|(x-y,x-z)|) \leq (1+\epsilon|x-y|) (1+\epsilon|x-z|),$$
we get for $N_p=1-s_p$ and $N_q = 1-s_q$
\begin{equation*}
\begin{split}
& \left| \langle T_m(f,g),h\rangle\right| \\
&\;\;\;\;\;\;\;\;\; \lesssim \sum_{i}\int_\R \left(\int_{\R} \frac{\epsilon}{(1+\epsilon|x-y|)^{N_p}} |\Delta_{I_i^1} f(y)| dy \right) \left(\int_{\R} \frac{\epsilon}{(1+\epsilon|x-z|)^{N_q}} |\Delta_{I_i^2} g(z)| dz \right) | \Delta_{I_i^3} h(x)| dx \\
&\;\;\;\;\;\;\;\;\; \lesssim  \epsilon^{2s} \sum_{i} \int_\R I_{s_p}(|\Delta_{I_i^1} f|)(x) I_{s_q}( |\Delta_{I_i^2} g|)(x) | \Delta_{I_i^3} h(x)| dx \\
&\;\;\;\;\;\;\;\;\; \lesssim  \epsilon^{2s}  \int_\R   \left( \sum_{i} I_{s_p}(|\Delta_{I_i^1} f|)(x)^2 \right)^ {1/2} \left( \sum_{i\in} I_{s_q}( |\Delta_{I_i^2} g|)(x)^2 \right) ^{1/2}  \sup_{i}| \Delta_{I_i^3} h(x)| dx,
\end{split}
\end{equation*}
where $I_{s_p}$ is the fractional integral operator of order $s_p$ (see Definition \ref{def:I}). Using that the maximal operator is bounded by the Hardy-Littlewood maximal function, we deduce that
$$ \left| \langle T_m(f,g),h\rangle\right| \lesssim  \epsilon^{2s}  \int_\R   \left( \sum_{i} I_{s_p} (|\Delta_{I_i^1} f|)(x)^2 \right)^ {1/2} \left( \sum_{i} I_{s_q} ( |\Delta_{I_i^2} g|)(x)^2 \right) ^{1/2}  \M(h) (x) dx.$$
Then by H\"older inequality with the exponents $p_s,q_s$ such that
$$ \frac{1}{r'} = \frac{1}{p}+\frac{1}{q} -2s = \left(\frac{1}{p}-s_p \right) + \left(\frac{1}{q} -s_q\right) :=\frac{1}{p_s}+\frac{1}{q_s},$$ it follows (by boundedness of $\M$ over Lebesgue spaces) that
$$ \left| \langle T_m(f,g),h\rangle\right| \lesssim  \epsilon^{2s}  \left\| \left( \sum_{i} I_{s_p}(|\Delta_{I_i^1} f|)^2 \right)^ {1/2} \right\|_{L^{p_s}} \left\| \left( \sum_{i} I_{s_q}( |\Delta_{I_i^2} g|)^2 \right) ^{1/2} \right\|_{L^{q_s}}  \| h \|_{L^{r}}.$$
As a consequence, 
\begin{align*}
\left| \langle T_m(f,g),h\rangle\right| & \lesssim  \epsilon^{2s}  \left\| \left( \sum_{i} I_{s_p}(|\Delta_{I_i^1} f|)^2 \right)^ {1/2}\right\|_{L^{p_s}} \left\| \left( \sum_{i} I_{s_q}( |\Delta_{I_i^2} g|)^2 \right) ^{1/2}\right\|_{L^{q_s}}  \| h\|_{L^{r}}.
\end{align*}
Thanks to Proposition \ref{prop:fracint}, we know that the fractional integral operator $I_{s_p}$ (resp.  $I_{s_q})$ is bounded from $L^p$ to $L^{p_s}$ (resp. from $L^q$ to $L^{q_s}$). Then it admits an $l^2$-valued extension (see Theorem 4.5.1 in \cite{Grafakos} and the original work of Marcinkiewicz and Zygmund \cite{MZ}). Consequently,
$$ A(3) \lesssim  \epsilon^{2s}  \left\| \left( \sum_{i} |\Delta_{I_i^1} f|^2 \right)^ {1/2}\right\|_{L^p} \left\| \left( \sum_{i} |\Delta_{I_i^1} g|^2 \right) ^{1/2}\right\|_{L^q}  \| h\|_{L^{r}}.$$
Recall that $(I_i^1)$ and $(I_i^2)$ form a bounded covering. We can thus apply Rubio de Francia's result (see Proposition \ref{prop:square}) for $p,q\geq 2$ to obtain
$$ \left| \langle T_m(f,g),h\rangle\right| \lesssim  \epsilon^{2s}  \left\|f \right\|_{L^p} \left\|  g \right\|_{L^q}  \| h\|_{L^{r}}.$$
Let us now deal with the limit case $s=0$. In this particular situation, argue similarly, replacing the fractional integrations operators $I_{s_p}$ and $I_{s_q}$ by the Hardy-Littlewood maximal operator $\M$. We let the reader check that everything still works with this minor modification since the Hardy-Littlewood maximal function admits $l^2$-valued extension too (see Theorem 4.6.6 in \cite{Grafakos} and the original work of Fefferman and Stein \cite{FS}).
\end{proof}

Remark~\ref{jaguar} implies the following corollary.

\begin{prop} \label{prop:L2bis} Let $\Gamma$ be a smooth and compact curve, and $m_\epsilon \in \mathcal{M}_\epsilon^\Gamma$. For exponents $p,q,r \in [1,\infty)$, there exists a constant $C=C(p,q,r)$ such that
$$ \left\|T_{m_\epsilon}(f,g) \right\|_{L^{r'}} \leq C \epsilon^ {\frac{1}{\max\{p,2\}}+\frac{1}{\max\{q,2\}}+\frac{1}{\max\{r,2\}}-1 } \|f\|_{L^p} \| g\|_{L^q}$$
as soon as 
$$ 2s:= \frac{1}{\max\{p,2\}}+\frac{1}{\max\{q,2\}}+\frac{1}{\max\{r,2\}} -1 \geq 0.$$
\end{prop}

Finally, when the curve $\Gamma$ is supposed to be nowhere characteristic, we can improve the decay in $\epsilon$ in the non local-$L^2$ case.

\begin{prop} \label{prop:L22} Consider $\Gamma$ a smooth and compact curve, which is nowhere characteristic (see Definition \ref{def:noncar}).  For exponents $p,q,r\in (1,\infty)$ verifying (\ref{eq:subholder}) and $\min\{p,q,r\}<2$, there exists a constant $C=C(p,q,r)$ such that
$$ \left\|T_{m_\epsilon}(f,g) \right\|_{L^{r'}} \leq C \epsilon^{\rho } \|f\|_{L^p} \| g\|_{L^q}$$
with 
$$ \rho:=\min\left\{ \frac{1}{\max\{p,2\} }+\frac{1}{\max\{q,2\} } +\frac{1}{\max\{r,2\} } -1 +\left( \max\{\frac{1}{p},\frac{1}{q},\frac{1}{r}\} - \frac{1}{2}\right) , 1\right\}.$$
\end{prop}

Since $\min\{p,q,r\}<2$ then $\left( \max\{\frac{1}{p},\frac{1}{q},\frac{1}{r}\} - \frac{1}{2}\right)$ is non negative so the new exponent $\rho$ is bigger than the one given by the previous proposition.

\begin{proof} First assume that only one of the three exponents $p,q,r$ is lower than $2$. Since $p,q,r$ play a symmetrical role, assume that $p=\min\{p,q,r\}\in(1,2)$.  
The proof is exactly the same as for Proposition \ref{prop:L2}, with the following modification. Since the curve $\gamma$ is supposed to be nowhere characteristic, then we can perform the decomposition  (explained in Lemma \ref{lem:covering}) only for $k=1$. The Proposition follows by Remark~\ref{jaguar}.
\end{proof}

\section{Study of particular points}

\label{SOPP}

\subsection{The point $(1,1,2)$} \label{subsec:112}

\begin{prop} \label{prop:112}
Let $m_\epsilon \in \mathcal{M}^\Gamma_\epsilon$. Then
\begin{itemize}
\item[(i)] If $\Gamma$ is nowhere characteristic,
$$
\left\| B_{m_\epsilon} (f,g) \right\|_{L^2} \lesssim \epsilon \|f\|_{L^1} \|g\|_{L^1}.
$$
\item[(ii)] If $\Gamma$ has a non vanishing curvature,
$$
\left\| B_{m_\epsilon} (f,g) \right\|_{L^2} \lesssim \epsilon \sqrt{- \log \epsilon} \|f\|_{L^1} \|g\|_{L^1}.
$$
and we have the restricted type inequality: for any three sets $F$, $G$, and $H$,
$$
\left| \left< B_{m_\epsilon} (\chi_F,\chi_G)\,,\,\chi_H \right> \right| \lesssim \epsilon |F| |G| |H|^{1/2}.
$$ 

\item[(iii)] If $\Gamma$ is arbitrary,
$$
\left\| B_{m_\epsilon} (f,g) \right\|_{L^2} \lesssim \sqrt{\epsilon} \|f\|_{L^1} \|g\|_{L^1}.
$$
\end{itemize}
Furthermore, the above estimates are optimal in that the powers of $\epsilon$ cannot be improved.
\end{prop}

\begin{proof} \underline{The $TT^*$ argument.} It is the first step of the proof:
\begin{equation}
\label{casoar}
\begin{split}
\left\| B_{m_\epsilon} (f,g) \right\|_{L^2}^2 & = \left\| \mathcal{F} B_{m_\epsilon} (f,g) \right\|_2^2 \\
& = \int \int \int m_\epsilon (\xi-\eta,\eta) \widehat{f}(\xi - \eta) \widehat{g}(\eta) \overline{m_\epsilon} (\xi-\zeta,\zeta) \overline{\widehat{f}} (\xi-\zeta) \overline{\widehat{g}}(\zeta) \,d\eta\,d\zeta\,d\xi \\
& = \int \int \int \int f(x^1) \overline f(x^2) g(y^1) \overline g(y^2) K_\epsilon(x^1-y^1,x^2-y^2,-x^1-x^2)\,dx^1\,dx^2\,dy^1\,dy^2,
\end{split}
\end{equation}
where
$$
K_\epsilon(a,b,c) = \int \int \int m_\epsilon(\xi-\eta,\eta) \overline{m_\epsilon}(\xi-\zeta,\zeta) e^{i\eta a + \zeta b + \xi c} d\eta \, d\zeta \, d\xi.
$$
In other words, $K_\epsilon = \mathcal{F}^{-1} (m_\epsilon(\xi-\eta,\eta) \overline{m_\epsilon}(\xi-\zeta,\zeta))$, if one views $m_\epsilon(\xi-\eta,\eta) m_\epsilon(\xi-\zeta,\zeta)$ as a function of $\eta,\zeta,\xi$. 

\bigskip

\noindent  \underline{Upper bounds for $B_{m_\epsilon}:L^1 \times L^1 \rightarrow L^2$.}
It is clear from (\ref{casoar}) that
$$
\left\| B_{m_\epsilon} (f,g) \right\|_{L^2}^2 \leq \|K_\epsilon\|_{L^\infty(\R^3)} \|f\|_{L^1}^2 \|g\|_{L^1}^2.
$$
Hence
$$
\|B_{m_\epsilon}\|_{L^1 \times L^1 \rightarrow L^2}^2 \leq \|K_\epsilon\|_{L^\infty(\R^3)} \leq \left\| m_\epsilon(\xi-\eta,\eta) m_\epsilon(\xi-\zeta,\zeta) \right\|_{L^1(\R^3)} = \int \left( \int \left|m_\epsilon(\xi - \eta,\eta)\right| \,d\eta \right)^2 d\xi.
$$
We now distinguish between the three cases of the theorem:
\begin{itemize}
\item If $\Gamma$ is non characteristic, then for any $\xi$, $\int m_\epsilon(\xi - \eta,\eta)\,d\eta \lesssim \epsilon$ thus
$$
\int \left( \int \left| m_\epsilon(\xi - \eta,\eta) \right| \,d\eta \right)^2 d\xi \lesssim \epsilon^2 
$$
and $\|B_{m_\epsilon}\|_{L^1 \times L^1 \rightarrow L^2} \lesssim \epsilon$.

\item If $\Gamma$ is arbitrary, then Cauchy-Schwarz gives
$$
\int \left( \int \left| m_\epsilon(\xi - \eta,\eta) \right| \,d\eta \right)^2 d\xi \lesssim \int \int \left| m_\epsilon(\xi - \eta,\eta)\right| ^2 \,d\eta \, d\xi \lesssim \epsilon
$$
which implies $\|B_{m_\epsilon}\|_{L^1 \times L^1 \rightarrow L^2} \lesssim \sqrt{\epsilon}$.

\item If $\Gamma$ has a non-vanishing curvature, the estimate is a little more involved. It is easy to see that one can restrict to regions where $\Gamma$ is parameterized as $\xi = \Phi (\eta)$. Then $m_\epsilon(\xi-\eta,\eta)$ is localized $\epsilon$ away from $\xi = \eta + \Phi(\eta)$. Difficulties appear where $\Phi'(\eta) = -1$; let us assume, without loss of generality, that $\Phi'(0)=-1$. Picking $C_0$ big enough,  and $\delta$ small enough, a small computation shows that for $|\xi| \leq C_0 \epsilon$, $\int m_\epsilon(\xi - \eta,\eta)\,d\eta \lesssim \sqrt{\epsilon}$, whereas for $\delta \geq \xi \geq C_0 \epsilon$, $\int m_\epsilon(\xi - \eta,\eta)\,d\eta \lesssim \frac{\epsilon}{\sqrt{|\xi|}}$. Thus
$$
\int_{|\xi|\leq \delta} \left( \int m_\epsilon(\xi - \eta,\eta)\,d\eta \right)^2 d\xi = \int_{|\xi| \leq C_0 \epsilon} \epsilon\,d\xi + \int_{\delta \geq \xi \geq C_0 \epsilon} \frac{\epsilon}{\xi} \,d\xi \lesssim \epsilon |\log (\epsilon)|.
$$
This gives $\|B_{m_\epsilon}\|_{L^1 \times L^1 \rightarrow L^2} \lesssim \sqrt \epsilon \sqrt{- \log \epsilon}$.
\end{itemize}

\bigskip

\noindent  \underline{Optimality} It follows by Proposition~\ref{macareux} and Proposition~\ref{fourmi}.
\end{proof}

\subsection{The point $(2,2,1)$} \label{subsec:221}

\begin{prop}  \label{prop:221} Let $m_\epsilon \in \mathcal{M}_\Gamma^\epsilon$.
\begin{itemize}
\item[(i)] If $\Gamma$ is nowhere characteristic,
$$
\left\| B_{m_\epsilon} (f,g) \right\|_{L^\infty} \lesssim \epsilon \|f\|_{L^2} \|g\|_{L^2}.
$$
\item[(ii)] If $\Gamma$ has a non vanishing curvature,
$$
\left\| B_{m_\epsilon} (f,g) \right\|_{L^\infty} \lesssim \epsilon^{3/4} \|f\|_{L^2} \|g\|_{L^2}.
$$
\item[(iii)] If $\Gamma$ is arbitrary,
$$
\left\| B_{m_\epsilon} (f,g) \right\|_{L^\infty} \lesssim \sqrt{\epsilon} \|f\|_{L^2} \|g\|_{L^2}.
$$
\end{itemize}
Furthermore, the above exponents of $\epsilon$ are optimal.
\end{prop}

\begin{proof}
\underline{Proof of $(i)$: $\Gamma$ non characteristic.} Define $\Gamma_\epsilon$ to be an $\epsilon$-neighbourhood of $\Gamma$. We split $\Gamma_\epsilon$ by considering its intersection with strips $\{(\xi,\eta)\,,\,n\epsilon < \eta \leq (n+1)\epsilon \}$ where $n \in \mathbb{Z}$ - of course, only finitely many of these intersections are non empty. Since $\Gamma$ is non characteristic, it is possible to write
$$
\Gamma_\epsilon \cap \{(\xi,\eta)\,,\,n\epsilon < \eta \leq (n+1)\epsilon \} \subset \{(\xi,\eta)\,,\,n\epsilon < \eta \leq (n+1)\epsilon\;\mbox{and}\; x_n^\epsilon - C_0 \epsilon < \xi \leq x_n^\epsilon + C_0 \epsilon\}
$$
where $(x_n^\epsilon)$ is a family of real numbers; $C_0$ a constant independent of $\epsilon$; the above decomposition is almost orthogonal in $\xi$, as it obviously is in $\eta$: there exists a constant $M$, also independent of $\epsilon$, such that at most $M$ intervals $[x_n^\epsilon - C_0 \epsilon,x_n^\epsilon + C_0 \epsilon]$ can have a non empty intersection.
Then by Cauchy-Schwarz and the almost orthogonality property,
\begin{equation*}
\begin{split}
\left| B_{m_\epsilon}(f,g)(x) \right| & = \left| \int \int e^{ix(\xi+\eta)} m_\epsilon(\xi,\eta) \widehat{f}(\xi) \widehat{g}(\eta)\,d\eta\,d\xi \right| \\
& \lesssim \int \int_{\Gamma_\epsilon} |\widehat{f}(\xi)| |\widehat{g}(\eta)| \,d\eta\,d\xi \\
& \lesssim \sum_n \int_{x_n^\epsilon-C_0 \epsilon}^{x_n^\epsilon+C_0 \epsilon} \int_{n\epsilon}^{(n+1)\epsilon} |\widehat{f}(\xi)| |\widehat{g}(\eta)| \,d\eta\,d\xi \\
& = \sum_n \int_{x_n^\epsilon-C_0 \epsilon}^{x_n^\epsilon+C_0 \epsilon} |\widehat{f}(\xi)| \,d\xi \int_{n\epsilon}^{(n+1)\epsilon} |\widehat{g}(\eta)| \,d\eta \\
& \lesssim \sum_n \sqrt{\epsilon} \| \widehat{f} \|_{L^2([x_n^\epsilon-C_0 \epsilon, x_n^\epsilon+C_0 \epsilon])} \sqrt{\epsilon} \left\|\widehat{g}\right\|_{L^2([n \epsilon,(n+1)\epsilon])} \\
& \lesssim \epsilon \left[ \sum_n \| \widehat{f} \|_{L^2([x_n^\epsilon-C_0 \epsilon, x_n^\epsilon+C_0 \epsilon])}^2 \right]^{1/2} \left[ \sum_n \left\|\widehat{g}\right\|_{L^2([n \epsilon,(n+1)\epsilon])}^2 \right]^{1/2} \\
& \lesssim \epsilon \|f\|_{L^2} \|g\|_{L^2}.
\end{split}
\end{equation*}
The norm $\epsilon$ for this bilinear operator is of course optimal by Proposition~\ref{fourmi} $(i)$.

\bigskip

\noindent  \underline{Proof of $(ii)$: $\Gamma$ has non vanishing curvature.} The proof of $(i)$ is valid except where $\Gamma$, in $(\xi,\eta)$ coordinates, has a tangent which is parallel to the $\xi$ or $\eta$ axes. By symmetry it suffices to focus on the former possibility, and assume that $\Gamma$ can be parameterized by $\eta = \phi(\xi)$, the problem being to treat regions where $\phi'$ vanishes. Without loss of generality, let us assume that $\phi'$ remains small, say less than $1/10$; it means that $\Gamma_\epsilon$ is contained in $\{ |\phi(\xi)-\eta| < 3 \epsilon \}$.
Proceeding as above, we split $\Gamma_\epsilon$ by considering its intersection with strips $\{(\xi,\eta)\,,\,n\epsilon < \eta \leq (n+1)\epsilon \}$ where $n \in \mathbb{Z}$. These intersections can be covered as follows:
\begin{align*}
\Gamma_\epsilon \cap \{n\epsilon < \eta \leq (n+1)\epsilon \} & \subset \{(\xi,\eta)\,,\,n\epsilon < \eta < (n+1)\epsilon\;\mbox{and}\;(n-1)\epsilon < \phi(\xi) < (n+2)\epsilon \} \\
&  := 
(x_n^\epsilon, y_n^\epsilon) \times (n\epsilon,(n+1) \epsilon).
\end{align*}
The almost orthogonality property for the intervals $(x_n^\epsilon,y_n^\epsilon)$ is obvious from their definition. Furthermore, since the curvature of $\Gamma$ does not vanish, their size can be bounded by
$$
y_n^\epsilon - x_n^\epsilon \lesssim \sqrt{\epsilon}.
$$
It is then easy to follow the proof of $(i)$, and get the desired estimate; it is optimal by Proposition~\ref{fourmi} $(iii)$.

\bigskip

\noindent  \underline{Proof of $(iii)$: $\Gamma$ arbitrary.} By duality, it suffices to prove
$$
\left\|B_{m_\epsilon}(f,g)\right\|_{L^2} \lesssim \sqrt{\epsilon} \|f\|_{L^1} \|g\|_{L^2}.
$$
But this is a simple consequence of the Cauchy-Schwarz, Hausdorff-Young, and Plancherel inequalities:
\begin{equation*}
\begin{split}
\left\| B_{m_\epsilon}(f,g) \right\|_{L^2} & = \left\| \int m_\epsilon(\xi-\eta,\eta) \widehat{f}(\xi-\eta) \widehat{g}(\eta) \,d\eta\right\|_{L^2(\xi)} \\ 
& \leq \left\| \left\| \widehat{f}\right\|_{L^\infty}  \left[ \int m_\epsilon(\xi-\eta,\eta)^2 \,d\eta \right]^{1/2} \|\widehat{g}\|_{L^2} \right\|_{L^2(\xi)} \\
& \lesssim \|f\|_{L^1} \|g\|_{L^2} \left[ \int \int m_\epsilon(\xi-\eta,\eta)^2 \,d\eta \,d\xi \right]^{1/2} \lesssim \sqrt{\epsilon} \|f\|_{L^1} \|g\|_{L^2}.
\end{split}
\end{equation*}
This bound is optimal by Proposition~\ref{fourmi} $(iv)$.
\end{proof}

\subsection{The point $(\infty,1,2)$} \label{subsec:i12}

\begin{prop} \label{prop:i12} Let $m_\epsilon \in \mathcal{M}_\epsilon^\Gamma$.
\begin{itemize}
\item[(i)] If $\Gamma$ is nowhere characteristic,
$$
\left\| B_{m_\epsilon} (f,g) \right\|_{L^2} \lesssim \epsilon^{1/4} \|f\|_{L^\infty} \|g\|_{L^1}.
$$
\item[(ii)] If $\Gamma$ has a non vanishing curvature,
$$
\left\| B_{m_\epsilon} (f,g) \right\|_{L^2} \lesssim \epsilon^{1/4} \sqrt{-\log(\epsilon)} \|f\|_{L^\infty} \|g\|_{L^1}.
$$
\item[(iii)] If $\Gamma$ is arbitrary,
$$
\left\| B_{m_\epsilon} (f,g) \right\|_{L^2} \lesssim \|f\|_{L^\infty} \|g\|_{L^1}.
$$
\end{itemize}
Furthermore, the bounds $(ii)$ and $(iii)$ are optimal up to the logarithmic factor.
\end{prop}

\begin{proof} \noindent \underline{The $TT^*$ argument.}
Recall~(\ref{casoar}):
\begin{equation*}
\left\| B_{m_\epsilon} (f,g) \right\|_{L^2}^2 = \int \int \int \int f(x^1) \overline{f}(x^2) g(y^1) \overline{g}(y^2) K_\epsilon(x^1-y^1,y^2-x^2,x^2-x^1)\,dx^1\,dx^2\,dy^1\,dy^2,
\end{equation*}
where
$$
K_\epsilon(a,b,c) = \int \int \int m_\epsilon(\xi-\eta,\eta) \overline{m_\epsilon}(\xi-\zeta,\zeta) e^{i\eta a + \zeta b + \xi c} d\eta \, d\zeta \, d\xi.
$$
Thus
\begin{equation*}
\begin{split}
\left\| B_{m_\epsilon} (f,g) \right\|_{L^2}^2 & \leq \|f\|_{L^\infty}^2 \int \int |g(y^1)| |g(y^2)| \int \int \left| K_\epsilon(x^1-y^1,y^2-x^2,x^2-x^1) \right| \,dx^1\,dx^2\,dy^1\,dy^2, \\
& \lesssim  \left(\sup_{y^1,y^2} \int \int \left| K_\epsilon(x^1-y^1,y^2-x_2,x^2-x^1)\right| \,dx^1\,dx^2\right) \|f\|_{L^\infty}^2 \|g\|_{L^1}^2.
\end{split}
\end{equation*}
Everything now boils down to estimating
$$
\int \int | K_\epsilon(x^1-y^1,y^2-x^2,x^2-x^1) | \,dx^1\,dx^2.
$$
A change of variables gives
$$
K_\epsilon(x^1-y^1,y^2-x^2,x^2-x^1) = \int \int F_{y}(\alpha,\beta) e^{i\alpha(y^1-x^1)} e^{i\beta(x^2-y^2)} \,d\alpha d\beta = \widehat{F_y} (x^1-y^1,y^2-x^2)
$$
where
\begin{equation}
\label{lapin}
F_{y}(\alpha,\beta) = \int m_\epsilon (\alpha,\xi-\alpha) m_\epsilon (\beta,\xi-\beta) e^{-iy\xi} \,d\xi\;\;\;\;\mbox{and}\;\;\;\;y = y^1 - y^2.
\end{equation}
Combining the few last lines,
\begin{equation}
\label{lievre}
\left\| B_{m_\epsilon} (f,g) \right\|_{L^2}^2 \lesssim \left(\sup_y \|\widehat{F_y}\|_{L^1}\right) \|f\|_{L^\infty}^2 \|g\|_{L^1}^2.
\end{equation}

\bigskip

\noindent \underline{Proof of $(i)$: $\Gamma$ non characteristic.} If $\Gamma$ is non characteristic, the support of $(\alpha,\xi) \mapsto m_\epsilon(\alpha,\xi-\alpha)$ is contained in the set $\{ |\alpha - \phi(\xi)| \leq C_0 \epsilon \}$ for a certain invertible function $\phi$, and a constant $C_0$. Given the definition of $F_y$~(\ref{lapin}), this implies immediately that
$$
| \operatorname{Supp} F_y | \sim \epsilon \;\;\;\;\mbox{and}\;\;\;\;\|F_y\|_{L^\infty(\R^2)} \lesssim \epsilon.
$$
Furthermore, since taking derivatives of the symbol $m_\epsilon$ essentially amounts to multiplying it by $\frac{1}{\epsilon}$, we also obtain
$$
\| \nabla^2_{\alpha,\beta} F_{y} \|_{L^\infty(\R^2)} \lesssim \frac{1}{\epsilon}.
$$
Combining these two estimates with Plancherel's identity gives now:
\begin{equation*}
\begin{split}
\|\widehat{F_y}\|_{L^1(\R^2)} & \lesssim \| \widehat{F_y} \|_{L^2(\R^2)}^{1/2} \| |\cdot|^2 \widehat{F_y} \|_{L^2(\R^2)}^{1/2} = \|F_y\|_{L^2(\R^2)}^{1/2} \left\| \nabla^2_{\alpha,\beta} F_{y} \right\|_{L^2(\R^2)}^{1/2} \\
& \lesssim \left[ | \operatorname{Supp} F_y |^{1/2} \|F_y\|_{L^\infty(\R^2)} | \operatorname{Supp} F_y |^{1/2} \| \nabla^2_{\alpha,\beta} F_{y}\|_{L^\infty(\R^2)} \right]^{1/2} \lesssim \sqrt{\epsilon}.
\end{split}
\end{equation*}
By~(\ref{lievre}), this gives the desired bound.

\bigskip

\noindent \underline{Proof of $(ii)$: $\Gamma$ has non vanishing curvature.} Consider the curve $\Gamma$ in the coordinates $(\alpha,\xi-\alpha)$. In regions where it can be parameterized by $\alpha = \phi(\xi)$, with $\phi$ smooth, and with a smooth inverse, the result follows from $(i)$. Difficulties appear when the parameterization becomes $\alpha = \phi(\xi)$, with $\phi'$ vanishing, or $\xi = \psi(\alpha)$, with $\psi'$ vanishing. We focus on these two cases from now on. In both cases, we assume for simplicity that the vanishing occurs at $0$: $\phi'(0)=0$ and $\psi'(0)=0$. Since the curvature does not vanish, $\phi''(0)$ as well as $\psi''(0)$ are non zero. Focusing on a small neighborhood $[-\delta,\delta]$ of $0$ in both cases, we will simply consider that $\phi(\xi) = \xi^2$, $\psi(\alpha)=\alpha^2$: it makes notations lighter, while retaining all the essential difficulties.

\begin{itemize} 
\item 
Let us start with the case where $\Gamma$ is parameterized by $\alpha = \xi^2$, where we restrict $\alpha$ to $[-\delta,\delta]$. The support of $(\alpha,\xi) \mapsto m_\epsilon(\alpha,\xi-\alpha)$ is then contained in the set $\{\alpha \in [-\delta,\delta], |\alpha - \xi^2| \leq 2 \epsilon \}$. For fixed $\alpha$, the set $\{\xi\,,\,|\xi^2-\alpha|<2 \epsilon\}$ has size 
\begin{equation*}
| \{\xi\,,\,|\xi^2-\alpha|< 2 \epsilon\} | \lesssim \left\{ \begin{array}{ll} \sqrt{\epsilon} & \mbox{if $|\alpha| < 10 \epsilon$} \vsp \\ \frac{\epsilon}{\sqrt{|\alpha|}} & \mbox{if $|\alpha|>10\epsilon$} \vsp \end{array} \right..
\end{equation*}
This implies
\begin{equation*}
| F_y(\alpha,\beta) | \lesssim \left\{ \begin{array}{ll} \sqrt{\epsilon} & \mbox{if $|\alpha| < 10 \epsilon$} \vsp \\ \frac{\epsilon}{\sqrt{|\alpha|}} & \mbox{if $|\alpha|>10\epsilon$} \vsp \end{array} \right..
\end{equation*}
Since furthermore $|\alpha-\beta|< 3 \epsilon$ on the support of $F_y$, we obtain
\begin{equation*}
\begin{split}
\|F_y\|_2^2 & = \int \int |F_y(\alpha,\beta)|^2\,d\alpha\,d\beta \lesssim \int \epsilon \sup_{\beta} |F_y(\alpha,\beta)|^2 \,d\alpha \\
& \lesssim \int_{|\alpha|<10\epsilon} \epsilon^2\,d\alpha + \int_{10\epsilon<|\alpha|<\delta} \frac{\epsilon}{\sqrt{|\alpha|}}\,d\alpha \lesssim - \epsilon^3 \log(\epsilon).
\end{split}
\end{equation*}
\item Let us consider now the case where $\Gamma$ is parameterized by $\xi = \alpha^2$. The support of $(\alpha,\xi) \mapsto m_\epsilon(\alpha,\xi-\alpha)$ is then contained in the set $\{\alpha \in [-\delta,\delta], |\xi - \alpha^2| \leq 2 \epsilon \}$. An examination of the definition of $F_y$ reveals that
$$
\| F_y \|_{L^\infty(\R^2)} \lesssim \epsilon
$$
and that, for fixed $\alpha$, the set $\operatorname{Supp} F_y(\alpha,\cdot)$ has size
\begin{equation*}
| \operatorname{Supp} F_y(\alpha,\cdot) | \lesssim \left\{ \begin{array}{ll} \sqrt{\epsilon} & \mbox{if $|\alpha| < 10 \sqrt{\epsilon}$} \vsp \\ \frac{\epsilon}{|\alpha|} & \mbox{if $|\alpha|>10\epsilon$} \vsp \end{array} \right.. 
\end{equation*}
This implies immediately that
$$
|\operatorname{Supp} F_y| \lesssim -\epsilon \log (\epsilon),
$$
which gives in turn, recalling that $\| F_y \|_{L^\infty(\R^2)} \lesssim \epsilon$,
$$
\|F_y\|_{L^2(\R^2)}^2 \lesssim |\operatorname{Supp} F_y| \| F_y \|_{L^\infty(\R^2)}^2 \lesssim - \epsilon^3 \log(\epsilon).
$$
\end{itemize}
Thus we could prove in both cases that $\|F_y\|_{L^2(\R^2)}^2 \lesssim - \epsilon^3 \log(\epsilon)$. One can deduces similarly that $\|\nabla^2_{\alpha,\beta} F_{y} \|_{L^2(\R^2)}^2 \lesssim - \frac{\log(\epsilon)}{\epsilon}$. These two bounds imply $\|F_y\|_{L^1(\R^2)} \lesssim -\sqrt{\epsilon} \log(\epsilon)$, which is the desired bound. It is optimal by Proposition~\ref{fourmi} $(iii)$.

\bigskip

\noindent \underline{Proof of $(iii)$: $\Gamma$ arbitrary.} The $L^2$ norm of $F_y$ can be estimated by Cauchy-Schwarz' inequality:
\begin{align*}
\|F_y\|_{L^2(\R^2)}^2 & = \int \int \left( \int m_\epsilon (\alpha,\xi-\alpha) m_\epsilon (\beta,\xi-\beta) e^{-iy\xi} \,d\xi \right)^2 \,d\alpha \,d\beta \\
& \leq \left(\int  \int \left|m_\epsilon (\alpha,\xi-\alpha)\right|^2 \,d\xi d\alpha\right) \left( \int \int \left|m_\epsilon (\beta,\xi-\beta)\right|^2 \,d\xi d\beta\right) \\ 
& =  \left(\int \int \left|m_\epsilon (\alpha,\xi-\alpha)\right|^2 \,d\xi \,d\alpha\right)^2 \\
& \lesssim \epsilon^2.
\end{align*}
Recall that taking derivatives of the symbol $m_\epsilon$ essentially amounts to multiplying it by $\frac{1}{\epsilon}$. Therefore, proceeding as above one can prove
$$
\left\| \nabla^2_{\alpha,\beta} F_{y} \right\|_{L^2(\R^2)}^2 \lesssim \epsilon^2 \frac{1}{\epsilon^4} = \frac{1}{\epsilon^2}.
$$
Putting these two estimates together gives, with the help of Plancherel's identity:
$$
\|\widehat{F_y}\|_{L^1(\R^2)} \lesssim \| \widehat{F_y} \|_{L^2(\R^2)}^{1/2} \| |\cdot|^2 \widehat{F_y} \|_{L^2(\R^2)}^{1/2} = \|F_y\|_{L^2(\R^2)}^{1/2} \left\| \nabla^2_{\alpha,\beta} F_{y} \right\|_{L^2(\R^2)}^{1/2} \lesssim 1.
$$
By~(\ref{lievre}), this gives the bound that we sought. It is optimal by Proposition~\ref{fourmi} $(iv)$. \end{proof}

\subsection{The point $(1,1,\infty)$} \label{subsec:11i}

\begin{prop} \label{prop:11i}
\begin{itemize} Let $m_\epsilon \in \mathcal{M}^\Gamma_\epsilon$.
\item[(i)] If $\Gamma$ is nowhere characteristic,
$$
\left\| B_{m_\epsilon} (f,g) \right\|_{L^1} \lesssim \sqrt{\epsilon} \|f\|_{L^1} \|g\|_{L^1}.
$$
\item[(ii)] If $\Gamma$ has a non-vanishing curvature,
$$
\left\| B_{m_\epsilon} (f,g) \right\|_{L^1} \lesssim \sqrt{\epsilon} |\log \epsilon| \|f\|_{L^1} \|g\|_{L^1}.
$$
\item[(iii)] If $\Gamma$ is arbitrary,
$$
\left\| B_{m_\epsilon} (f,g) \right\|_{L^1} \lesssim \|f\|_{L^1} \|g\|_{L^1}.
$$
\end{itemize}
Furthermore, the $\epsilon$-dependence of these bounds are optimal up to the logarithmic factor.
\end{prop}

\begin{proof} \underline{Proof of $(i)$: $\Gamma$ non characteristic.} Recall~(\ref{physical})
$$
B_{m_\epsilon}(f,g) = \int \int  \widehat{m_\epsilon}(y-x,z-x) f(y) g(z) \,dy\,dz;
$$
this implies
$$
\left\| B_{m_\epsilon}(f,g) \right\|_{L^1} \leq \|f\|_{L^1} \|g\|_{L^1} \left(\sup_{y,z} \left\|\widehat{m_\epsilon}(y-\cdot,z-\cdot) \right\|_{L^1}\right).
$$
Thus it suffices to estimate $\sup_{y,z} \left\|\widehat{m_\epsilon} (y-\cdot,z-\cdot) \right\|_{L^1}$. In order to do so, write $\widehat{m_\epsilon}$ as
$$
\widehat{m_\epsilon}(y-x,z-x) = \int e^{i\xi(x-y)} \int e^{i\eta (x-z)} m_\epsilon(\xi,\eta) \,d\eta\,d\xi := \int e^{ix \alpha} F_{y,z}(\alpha) \,d\alpha,
$$
with
$$ F_{y,z}(\alpha):= \int_{\xi+\eta=\alpha} e^{-i(\xi y+\eta z)} m_\epsilon(\xi,\eta) \,d\eta\,d\xi.$$
Hence$$
\sup_{y,z} \left\|\widehat{m_\epsilon}(y-\cdot,z-\cdot) \right\|_{L^1} = \sup_{y,z} \left\| \widehat{F_{y,z}} \right\|_{L^1}.
$$
In order to estimate $\widehat{F_{y,z}}$ in $L^1$, we want to interpolate it between $L^2$ and $L^2(x^2 dx)$. Since $\Gamma$ is non characteristic and bounded,
$$
\left\| \widehat{F_{y,z}} \right\|_{L^2}^2 = \left\| F_{y,z} \right\|_{L^2}^2 \lesssim  \left\| F_{y,z} \right\|_{L^\infty}^2 \lesssim \epsilon^2.
$$
Similarly, using the fact that $\nabla m_\epsilon$ has size at most $\frac{1}{\epsilon}$, one finds
$$
\left\| x \widehat{F_{y,z}}(x) \right\|_{L^2}^2 = \left\| \partial_\xi F_{y,z} \right\|_{L^2}^2  \lesssim \epsilon^2 \frac{1}{\epsilon^2}.
$$
This gives the estimate since
$$
\sup_{y,z} \left\|\widehat{m_\epsilon}(y-\cdot,z-\cdot) \right\|_{L^1} \lesssim \sup_{y,z} \left\| \widehat{F_{y,z}} \right\|_{L^1} \lesssim \|\widehat{F_{y,z}}\|_{L^2}^{1/2} \|x \widehat{F_{y,z}}(x)\|_{L^2}^{1/2}\lesssim \sqrt{\epsilon}.
$$
Optimality is a consequence of Proposition~\ref{macareux}.

\bigskip \noindent \underline{Proof of $(ii)$: $\Gamma$ has a non-vanishing curvature.} Proceeding as above, things boil down to estimating $\|\widehat{F_t} \|_{L^1}$ for every $t=(y,z)$; as above, we will obtain this estimate by interpolating $L^1$ between $L^2$ and $L^2(x^2 dx)$. 

Treating the parts of $\Gamma$ which are non-characteristic can be done by using the previous case. We now focus on a part of $\Gamma$ which is characteristic, namely it has a tangent parallel to the $(\xi-\eta)$ axis. For the sake of simplicity, we just consider a model case: $\Gamma$ will be given (say in the ball of radius $1$) around $(0,0)$ by the equation $(\xi+\eta)^2 = (\xi - \eta)^2$. Next, we denote $\Gamma_\epsilon$ for the set of points which are within $\epsilon$ of $\Gamma$, and $\mathcal{D}_\alpha$ for the line given by the equation $\xi+\eta = \alpha$.

The formula giving $\F_{(y,z)}$ implies immediately that
$$
|F_{(y,z)}(\alpha)| \leq |\mathcal{D}_\alpha \cup \Gamma_\epsilon| \lesssim \left\{ \begin{array}{ll} \sqrt{|\alpha|} & \mbox{if $|\alpha|\leq 100 \epsilon$} \\ \frac{\epsilon}{\sqrt{|\alpha}|} & \mbox{otherwise} \end{array} \right.
$$
Thus by Plancherel's inequality
$$
\left\| \widehat{F_{y,z}} \right\|_{L^2}^2 = \left\|F_{y,z}(\alpha)\right\|_2^2 \leq \int_{|\alpha| \leq 1} |\mathcal{D}_\alpha \cup \Gamma_\epsilon|^2 \,d\alpha \lesssim \epsilon^2 |\log \epsilon|.
$$
One finds as above
$$
\|x \widehat{F_t}\|_{L^2}^2 = \|\partial_\xi F_t \|_{L^2}^2 \lesssim  |\log \epsilon|,
$$
and the result follows by interpolation. It is optimal up to the logarithmic factor by Proposition~\ref{macareux}.

\bigskip \noindent \underline{Proof of $(iii)$: arbitrary $\Gamma$.} Still following the above pattern, we get by Cauchy-Schwarz
$$
\|\widehat{F_t}\|_{L^2}^2 \lesssim \|F_t\|_{L^2}^2 \leq \int \left| \int  |m_\epsilon(\xi-\eta,\eta)| \,d\eta\right|^2 \,d\xi \lesssim \int \int \left| m_\epsilon(\xi,\eta) \right|^2 \,d\eta\,d\xi \lesssim \epsilon.
$$
Similarly,
$$
\|x \widehat{F_t}\|_{L^2}^2 = \|\partial_\xi F_t \|_{L^2}^2 \lesssim \frac{1}{\epsilon}.
$$
This gives
$$
\|x \widehat{F_t}\|_{L^1} \lesssim \|\widehat{F_t}\|_{L^2}^{1/2}\|x \widehat{F_t}(x)\|_{L^2}^{1/2} \lesssim 1
$$
which allows us to conclude the proof. Optimality follows from Proposition~\ref{fourmi} $(iv)$.
\end{proof}

\subsection{The point $(1,1,1)$} \label{subsec:111}

\begin{prop} \label{prop:111}
For an arbitrary $\Gamma$,
$$
\left\| B_{m_\epsilon} (f,g) \right\|_{L^\infty} \lesssim \epsilon \|f\|_{L^1} \|g\|_{L^1},
$$
and the $\epsilon$-dependence of the bound is optimal.
\end{prop}

\begin{proof} The optimality follows from Proposition~\ref{fourmi} $(i)$. To prove that the bound holds, recall~(\ref{physical}), which gives
$$
< B_{m_\epsilon}(f,g),h> = \int \int \int \widehat{m_\epsilon}(y-x,z-x) f(y) g(z) h(x)\,dx\,dy\,dz.
$$
Therefore,
$$
\left| < B_{m_\epsilon}(f,g),h> \right| \lesssim \left\| \widehat{m_\epsilon} \right\|_{L^\infty(\R^2)} \|f\|_{L^1} \|g\|_{L^1} \|h\|_{L^1},
$$
and
$$
\left\| \widehat{m_\epsilon} \right\|_{L^\infty(\R^2)} \lesssim \left\| m_\epsilon \right\|_{L^1(\R^2)} \lesssim \epsilon.
$$
The optimality comes from Proposition \ref{fourmi}.
\end{proof}

\section{Close to H\"older points, in the non-vanishing curvature case}

\label{CTHP}

In this section, we examine the case of Lebesgue exponents $(p,q,r)$, with $\frac{1}{p}+\frac{1}{q}+\frac{1}{r}$ close to 1 when $\Gamma$ has a non-vanishing curvature. If all three exponents are larger than 2, this case is taken care of by Proposition~\ref{prop:L2}, and the assumption $m_\epsilon \in \mathcal{M}_\epsilon^\Gamma$ suffices. If one exponent is less than 2, it seems that more regularity is needed from $m_\epsilon$, namely that it belongs to $\mathcal{N}_\epsilon^\Gamma$. We will distinguish two cases: $(p,q,r) = (2,\infty,2)$; and $(p,q,r)$ close to $(\infty,\infty, 1)$. Interpolation will then give all Lebesgue exponents such that $\frac{1}{p}+\frac{1}{q}+\frac{1}{r}>1$, with an arbitrarily small deviation from the optimal bound $\epsilon^{\frac{1}{p}+\frac{1}{q}+\frac{1}{r}-1}$.

\subsection{The point $(2,2,\infty)$}

\begin{prop} Assume that $m_\epsilon$ belongs to $\mathcal{N}^\Gamma_\epsilon$, and that $\Gamma$ has a non-vanishing curvature. Then
$$
\|B_{m_\epsilon}(f,g)\|_{L^1} \lesssim \|f\|_{L^2} \|g\|_{L^2}.
$$
\end{prop}

\begin{proof}
\noindent
\noindent \underline{Step 1: decomposition of $m_\epsilon$.}
The proof of the proposition presents new difficulties when the tangent of $\Gamma$ is parallel to one of the coordinate axes; otherwise, it is possible to rely on Proposition~\ref{prop:L22}. 

For the sake of simplicity in the notations, we will only treat a model case. Namely, we shall assume that $\Gamma$ is the circle with radius one and center $(\xi=0,\eta = 1)$ i.e. $\Gamma$ is given by the equation $\xi^2 + (\eta-1)^2 = 1$. We shall focus on the tangency point of the circle with the $\xi$ axis, $(\xi=0,\eta=0)$: thus we can assume that $m_\epsilon=0$ if $|(\xi,\eta)| \geq \frac{1}{20}$.
Recall that the support of $m_\epsilon$ is contained in a strip of width $2\epsilon$ around $\Gamma$.

Next we split smoothly $m_\epsilon$ into a sum of symbols each of which is supported on a chord of length $\frac{1}{\sqrt{\epsilon}}$. Switch for a moment to polar coordinates with center $(\xi = 0,\eta=1)$, denoting $\theta$ for the angular coordinate, with the convention that $\theta = 0$ corresponds to the $\eta$ axis below $(0,1)$: $\{\xi=0,\eta\leq 1\}$. Next let $\Phi$ be a smooth function on $\mathbb{R}$, equal to $0$ outside of $[-2,2]$, equal to $1$ on $[-1,1]$, and such that $\sum_{n \in \mathbb{Z}} \Phi(\cdot-n) = 1$. Finally set
$$
m_\epsilon^k(\xi,\eta) := m_\epsilon(\xi,\eta) \Phi \left( \frac{\theta}{\sqrt{\epsilon}} - k \right)
$$
so that
$$
m_\epsilon(\xi,\eta) = \sum_{|k|\leq \frac{1}{10\sqrt{\epsilon}}} m_\epsilon^k(\xi,\eta)
$$
(notice that the above sum runs over $|k|\leq \frac{1}{10\sqrt{\epsilon}}$ due to our assumption that $m_\epsilon$ vanishes for $|(\xi,\eta)| \geq \frac{1}{20}$).
So each of the symbols is supported on a chord of length $\sim \sqrt{\epsilon}$, thickened to reach a width of length $\sim \epsilon$, and with an angular parameter $\theta \sim k \sqrt{\epsilon}$. Thus it suffices to control
\begin{equation*}
\begin{split}
B_{m_\epsilon}(f,g)(x) = \sum_k \int \int \widehat{m^k_\epsilon}(y-x,z-x)f(y)g(z)\,dy\,dz.
\end{split}
\end{equation*}
Now denote $I_k$, respectively $J_k$, the intervals given by the projection of the support of $m_\epsilon^k(\xi,\eta)$ on the $\xi$, respectively $\eta$ axis. It is easy to check that these intervals are almost orthogonal:
$$
\forall x , \;\;\;\;\;\sum_k \chi_{I_k}(x) \lesssim 1 \;\;\;\;\mbox{and} \;\;\;\;\sum_k \chi_{J_k}(x) \lesssim 1 
$$
(where the implicit constants do not depend on $k$). Define 
$$
f_k := \chi_{I_k}(D) f \;\;\;\;g_k := \chi_{J_k}(D) g.
$$
The quantity to control can thus also be written
\begin{equation}
\label{chouette}
B_{m_\epsilon}(f,g)(x) = \sum_k \int \int \widehat{m^k_\epsilon}(y-x,z-x)f_k(y)g_k(z)\,dy\,dz.
\end{equation}

\bigskip
\noindent \underline{Step 2: examination of the kernels.} We claim that the kernels $\widehat{m^k_\epsilon}$ are uniformly bounded in $L^1(\mathbb{R}^2$. By translation and rotation invariance, it suffices to see this for $\widehat{m^0_\epsilon}$. 
Then, with the notations of Definition \ref{herisson}, in the chord of length $\sqrt{\epsilon}$ around the point $(0,0)$ we can write
$$ \partial_\xi = \partial_{(\nabla \nu)^\perp} + O(\epsilon^{1/2}) \partial_{\nabla \nu}.$$
Since $m_\epsilon$ belongs to $\mathcal{N}^\Gamma_\epsilon$, we also have
$$
\left\| \partial_\xi^\alpha \partial_\eta^\beta m^0_\epsilon \right\|_{L^1(\R^2)} \lesssim \epsilon^{3/2} \epsilon^{-\frac{\alpha}{2} - \beta}.
$$
This gives on the Fourier side (keeping in mind that $\mathcal{F}$ maps $L^1$ to $L^\infty$)
\begin{equation}
\label{puffin}
\left| \widehat{m^0_\epsilon}(y,z) \right| \lesssim \epsilon^{3/2} \inf \left(1\,,\,\frac{1}{\left( \sqrt{\epsilon} |y| + \epsilon |z| \right)^N} \right)
\end{equation}
for any number $N$. This implies obviously the desired bound.

\bigskip
\noindent \underline{Step 3: orthogonality;} It suffices now to use that the $(I_k)$ and $(J_k)$ form a bounded covering of the real line to get
\begin{equation*}
\begin{split}
\left\|B_{m_\epsilon}(f,g)\right\|_1 & \lesssim \sum_k \left\| \int \int \widehat{m^k_\epsilon}(y-x,z-x)f_k(y)g_k(z)\,dy\,dz \right\|_{L^1(x)} \\
& \lesssim \sum_k \left\| \widehat{m^k_\epsilon} \right\|_{L^1(\mathbb{R}^2)} \left\|f_k\right\|_{L^2(I_k)} \left\|g_k\right\|_{L^2(j_k)} \\
& \lesssim \left( \sum_k \left\|f_k\right\|_{L^2(I_k)}^2 \right)^{1/2}  \left( \sum_k \left\|f_k\right\|_{L^2(J_k)}^2 \right)^{1/2} \\
& \lesssim \|f\|_{L^2} \|g\|_{L^2}.
\end{split}
\end{equation*}
\end{proof}

\subsection{Points close to $(\infty,\infty,1)$}

\begin{prop} Assume that $m_\epsilon$ belongs to $\mathcal{N}^\Gamma_\epsilon$, and that $\Gamma$ has a non-vanishing curvature. Then for exponents $p,q \geq 2$,
$$
\|B_{m_\epsilon}(f,g)\|_{L^\infty} \lesssim \epsilon^{\frac{3}{4q}+\frac{1}{2p}} |\log \epsilon| \|f\|_{L^p} \|g\|_{L^q}.
$$
\end{prop}

\begin{rem}
This proposition is interesting in the limit where $p$ and $q$ tend to $\infty$. One approaches the point $(p,q,r) = (\infty,\infty,1)$, with a bound $O(\epsilon^{\frac{3}{4q}+\frac{1}{2p}})$ which converges to the optimal one at the limit point, namely $O(1)$.
\end{rem}

\begin{proof}
\noindent
\underline{Step 1: decomposition of $m_\epsilon$.}
This step is identical to Step 1 of the previous proposition; so we simply adopt the notations defined there.
The only difference is that it now suffices to control (by translation invariance)
\begin{equation}
\label{chouettebis}
B_{m_\epsilon}(f,g)(0) = \sum_k \int \int \widehat{m^k_\epsilon}(y,z)f_k(y)g_k(z)\,dy\,dz.
\end{equation}

\bigskip \noindent
\underline{Step 2: reduction to simpler kernels.}
The choice of the length scale $\sqrt{\epsilon}$ ensures that $\widehat{m^k_\epsilon}$ is essentially supported on a rectangle. 
We shall establish this for $m^0_\epsilon$, the general case following by rotating the plane. Since $m^0_\epsilon$ belongs to $\mathcal{N}^\Gamma_\epsilon$, it satisfies
$$
\left\| \partial_\xi^\alpha \partial_\eta^\beta m^0_\epsilon \right\|_{L^1(\R^2)} \lesssim \epsilon^{3/2} \epsilon^{-\frac{\alpha}{2} - \beta}.
$$
This gives on the Fourier side (keeping in mind that $\mathcal{F}$ maps $L^1$ to $L^\infty$)
\begin{equation}
\label{puffinbis}
\left| \widehat{m^0_\epsilon}(y,z) \right| \lesssim \epsilon^{3/2} \inf \left(1\,,\,\frac{1}{\left( \sqrt{\epsilon} |y| + \epsilon |z| \right)^N} \right)
\end{equation}
for any number $N$. Denoting $F$ for the characteristic function of the unit cube, the above inequality implies that $\widehat{m^0_\epsilon}(y,z)$ can be bounded by
$$
\left| \widehat{m^0_\epsilon}(y,z) \right| \lesssim \epsilon^{3/2} \sum_{\ell \in \mathbb{N}} \alpha_\ell F(2^{-\ell} \sqrt{ \epsilon}y\,,2^{-\ell} \epsilon z)
$$
where the sequence $(\alpha_k)$ decays very fast. Denoting $R_\phi$ for the rotation of angle $\phi$ around $(0,0)$, one can show just like the above inequality that
$$
\left| \widehat{m^k_\epsilon}(y,z) \right| \lesssim \epsilon^{3/2} \sum_{\ell \in \mathbb{N}} \alpha_\ell F \left( R_{k \sqrt{\epsilon}} (2^{-\ell} \sqrt{ \epsilon}y\,,2^{-\ell} \epsilon z) \right).
$$
We see from~(\ref{chouettebis}) that
\begin{equation}
\label{mesangebis}
\left| B_{m_\epsilon}(f,g)(0) \right| \lesssim \epsilon^{3/2} \sum_k \sum_\ell \alpha_\ell \int F \left( R_{k \sqrt{\epsilon}} (2^{-\ell} \sqrt{\epsilon}y\,,\,2^{-\ell} \epsilon z) \right) |f_k(y)| \, |g_k(z)|\,dy\,dz.
\end{equation}

\bigskip \noindent \underline{Step 3: the crucial claim and why it implies the proposition.}
Let us denote from now on
$$
F^k_\epsilon(y,z) = F \left( R_{k \sqrt{\epsilon}} (\sqrt{\epsilon}y\,,\,\epsilon z) \right).
$$
We will prove the following claim, which will imply the proposition.
\begin{claim}
\label{autour}
For any sequence of functions $(f_k)$,
\begin{equation}
\label{hirondelle}
\left\| \left[ \sum_k \left| \int F^k_\epsilon(y,z) f_k(y) \,dy \right|^2 \right]^{1/2} \right\|_{L^{q'}(z)} \lesssim \epsilon^{\frac{3}{4q}+\frac{1}{2p}} |\log \epsilon| \left\| \left[ \sum_k f_k^2 \right]^{1/2} \right\|_{L^p}.
\end{equation}
\end{claim}
Why does this claim imply the proposition? 
Starting from~(\ref{chouettebis}) and using successively the Cauchy-Schwarz (in $k$) and H\"older (in $z$) inequalities; the above claim; and Rubio de Francia's inequality gives
\begin{equation}
\begin{split}
& \left| B_{m_\epsilon}(f,g)(0) \right| \\
& \;\;\;\;\; \lesssim \sum_k \sum_\ell \alpha_\ell \int \int F \left( R_{ k \sqrt{\epsilon}} (2^{-\ell} \sqrt{\epsilon}y\,,\,2^{-2\ell} \epsilon z) \right) |f_k(y)| \, |g_k(z)| \,dy\,dz \\
& \;\;\;\;\; \lesssim \sum_\ell \alpha_\ell  \left\| \left[ \sum_k \left| \int F \left( R_{k \sqrt{\epsilon}} (2^{-\ell} \sqrt{\epsilon}y\,,\,2^{-2 \ell} \epsilon z \right) |f_k(y)| \,dy \right|^2 \right]^{1/2} \right\|_{L^{q'}(z)} \left\| \left[ \sum_k g_k^2 \right]^{1/2} \right\|_{L^q} \\
& \;\;\;\;\; \lesssim \epsilon^{\frac{3}{4q}+\frac{1}{2p}} |\log \epsilon| \left\| \left[ \sum_k f_k^2 \right]^{1/2} \right\|_{L^p} \left\| \left[ \sum_k g_k^2 \right]^{1/2} \right\|_{L^q} \\
& \;\;\;\;\; \lesssim \epsilon^{\frac{3}{4q}+\frac{1}{2p}} |\log \epsilon| \left\| f \right\|_{L^p} \left\|g\right\|_{L^q}.
\end{split}
\end{equation}
This is exactly the statement of the proposition.

\bigskip \noindent \underline{Step 4: decomposition of $f_k$ along its level sets.}
Write
$$
f_k(y) = \sum_j f_k^j(y)
$$
where $f_k^j(y)$ takes either values between $2^{j-1}$ and $2^j$, or equals $0$. We can a fortiori assume that $f_k^j$ either takes the value $2^j$, or equals $0$. In other words, we will assume that
$$
f_k^j = 2^j \chi_{E^k_j}
$$
for some set $E^k_j$. Observe that for any $z$, $\operatorname{Supp} F^k_\epsilon(\cdot,z) \subset \left[-\frac{C_0k}{\sqrt{\epsilon}}\,,\,\frac{C_0k}{\sqrt{\epsilon}}\right]$ for a constant $C_0$. Thus we can assume
$$
E^k_j \subset \left[ -\frac{k}{\sqrt{\epsilon}}\,,\,\frac{k}{\sqrt{\epsilon}} \right],
$$
for the parts of $E^k_j$ outside of $ \left[ -\frac{k}{\sqrt{\epsilon}}\,,\,\frac{k}{\sqrt{\epsilon}} \right]$ do not contribute to~(\ref{hirondelle}).
We will need a bound on $\sum_{k=1}^n |E^j_k|$, which we now derive:
\begin{align}
\sum_{k=-n}^n |E^j_k| & = \sum_{k=-n}^n \int_{-\frac{C_0 n}{\sqrt{\epsilon}}}^{\frac{C_0 n}{\sqrt{\epsilon}}} \chi_{E^j_k}(y)\,dy = 2^{-2j} \int_{-\frac{C_0 n}{\sqrt{\epsilon}}}^{\frac{C_0 n}{\sqrt{\epsilon}}} \sum_{k=-n}^n f^j_k (y)^2\,dy \nonumber \\
&  \lesssim 2^{-2j} \left\| \left(\sum_k (f_k^j)^2 \right)^{1/2} \right\|_{L^p}^2 \left( \frac{n}{\sqrt{\epsilon}} \right)^{1-\frac{2}{p}}. \label{eq:sumei}
\end{align}
Finally, observe that it suffices to prove the Claim~\ref{autour} when $f_k$ is replaced by $f^j_k$. Indeed, the scales $j$ such that $2^j < \epsilon^{100}$ can be estimated trivially, whereas summing over the other scales simply contributes $\log \epsilon$. Thus it suffices to prove
$$
\left\| \left[ \sum_k \left| \int F^k_\epsilon(y,z) f_k^j(y) \,dy \right|^2 \right]^{1/2} \right\|_{L^{q'}(z)} \lesssim \epsilon^{\frac{3}{4q}+\frac{1}{2p}} \left\| \left[ \sum_k f_k^2 \right]^{1/2} \right\|_{L^p}
$$
in order to deduce the claim.

\bigskip \noindent \underline{Step 5: discretization of the $z$ variable.} The variable $z$ in~(\ref{hirondelle}) can be restricted to $|z| \in \left[ -\frac{2}{\epsilon}\,,\,\frac{2}{\epsilon} \right]$, for otherwise $F_\epsilon^k(y,z)$ vanishes. We now split the interval $\left[ -\frac{2}{\epsilon}\,,\,\frac{2}{\epsilon} \right]$ into $\sim \frac{1}{\sqrt{\epsilon}}$ intervals $\left[ \frac{Z}{\sqrt{\epsilon}} \, , \, \frac{Z+1}{\sqrt{\epsilon}}\right]$ ($Z$ is an integer such that $|Z| \leq \frac{2}{\sqrt{\epsilon}}$).

Observe that if $z \in \left[ \frac{Z}{\sqrt{\epsilon}} \, , \, \frac{Z+1}{\sqrt{\epsilon}}\right]$, $\operatorname{Supp}_y F_\epsilon^k(y,z) \subset I_{\epsilon}^{k,z}$ where
$$
I_{\epsilon}^{k,z} := \left[ \tan(k\sqrt{\epsilon}) \frac{Z}{\sqrt{\epsilon}} - \frac{C_0}{\epsilon}\,,\,\tan(k\sqrt{\epsilon}) \frac{Z}{\sqrt{\epsilon}} + \frac{C_0}{\epsilon} \right].
$$
for a sufficiently big constant $C_0$. Thus, if $z \in \left[ \frac{Z}{\sqrt{\epsilon}} \, , \, \frac{Z+1}{\sqrt{\epsilon}}\right]$,
$$
\left| \int F^k_\epsilon(y,z) \chi_{E}(y) \,dy \right| \lesssim \epsilon \sqrt{\epsilon} \left| I_{\epsilon}^{k,z} \cap E_j^k \right|.
$$
This implies that
\begin{equation}
\label{lapinbis}
\begin{split}
\left\| \left[ \sum_k \left| \int F^k_\epsilon(y,z) f_k^j(y) \,dy \right|^2 \right]^{1/2} \right\|_{L^{q'}(z)} 
& \lesssim \epsilon \sqrt{\epsilon} 2^j \left[ \sum_{|Z|\leq \frac{2}{\sqrt{\epsilon}}} \frac{1}{\sqrt{\epsilon}} \left( \sum_k \left| I_{\epsilon}^{k,Z} \cap E^j_k \right|^2 \right)^{\frac{q'}{2}} \right]^{\frac{1}{q'}} \\
& = \epsilon^{\frac{3}{2}-\frac{1}{2q'}} 2^j \left[ \sum_Z \left\| \left| I_{\epsilon}^{k,Z} \cap E^j_k \right| \right\|_{\ell^2(k)}^{q'} \right]^{\frac{1}{q'}}.
\end{split}
\end{equation}

\bigskip \noindent \underline{Step 6: proof of the claim.} We will bound the above right-hand side by interpolating the $\ell^2$ norm between $\ell^1$ and $\ell^\infty$. The $\ell^\infty$ bound is the easier one: since the number of indices $Z$ is of the order of $\frac{1}{\sqrt{\epsilon}}$, and the length of $I_{\epsilon}^{k,z}$ is bounded by $2C_0 \frac{1}{\sqrt{\epsilon}}$,
\begin{equation}
\label{souris}
\left[ \sum_Z \left\| \left| I_{\epsilon}^{k,Z} \cap E^j_k \right| \right\|_{\ell^\infty_k}^{q'} \right]^{\frac{1}{q'}} = \left[ \sum_Z \left( \sup_k \left| I_{\epsilon}^{k,Z} \cap E^j_k \right| \right)^{q'} \right]^{\frac{1}{q'}} \lesssim \left[ \frac{1}{\sqrt{\epsilon}} \left( \frac{1}{\sqrt{\epsilon}} \right)^{q'} \right]^\frac{1}{q'} \lesssim \epsilon^{-\frac{1}{2}-\frac{1}{2q'}}.
\end{equation}
For the $\ell^1$ bound, use first the embedding $\ell^1 \hookrightarrow \ell^q$,
\begin{equation}
\label{musaraigne}
\begin{split}
\left[ \sum_Z \left\| \left| I_{\epsilon}^{k,Z} \cap E^j_k \right| \right\|_{\ell^1_k}^{q'} \right]^{\frac{1}{q'}} = \left[ \sum_Z \left( \sum_k \left| I_{\epsilon}^{k,Z} \cap E^j_k \right| \right)^{q'} \right]^{\frac{1}{q'}} \lesssim \sum_Z \sum_k \left| I_{\epsilon}^{k,Z} \cap E^j_k \right|
\end{split}
\end{equation}
Next use that, given $k$, a number $x$ can belong to at most $\sim \frac{1}{k \sqrt{\epsilon}}$ intervals $I^{k,Z}_\epsilon$. This implies that $\sum_k |I_{\epsilon}^{k,Z} \cap E| \lesssim \frac{1}{k\sqrt{\epsilon}} |E|$. Thus the above can be bounded with (\ref{eq:sumei}) by 
\begin{equation}
\label{mulot}
\begin{split}
(\ref{musaraigne}) & \lesssim \sum_{|k|\leq\frac{1}{10 \sqrt{\epsilon}}} \frac{|E^j_k|}{k \sqrt{\epsilon}} \\
& \lesssim \frac{1}{\sqrt{\epsilon}} \sum_{n=0}^{\frac{1}{10\sqrt{\epsilon}}-1} \left[\left(\frac{1}{n} - \frac{1}{n+1} \right) \left(\sum_{k=0}^n |E^j_k|\right)\right] + \frac{\sqrt{\epsilon}}{\sqrt{\epsilon}} \sum_{k=0}^{\frac{1}{10 \sqrt{\epsilon}}} |E^j_k| \\
& \lesssim 2^{-2j} \left\| \left( \sum_k (f^j_k)^2 \right)^{1/2}\right\|_{L^p}^2 \left[ \frac{1}{\sqrt{\epsilon}} \sum_{n=0}^{\frac{1}{10\sqrt{\epsilon}}-1} \frac{1}{1+n^2} \left(\frac{n}{\sqrt{\epsilon}}\right)^{1-\frac{2}{p}} + \left(\frac{1}{\epsilon}\right)^{1-\frac{2}{p}} \right] \\
& \lesssim 2^{-2j} \left\| \left( \sum_k (f^j_k)^2 \right)^{1/2}\right\|_{L^p}^2 \epsilon^{-1+\frac{1}{p}}.
\end{split}
\end{equation}
Starting with the inequality~(\ref{lapinbis}), and interpolating between~(\ref{souris}) and~(\ref{mulot}) gives 
\begin{equation*}
\begin{split}
\left\| \left[ \sum_k \left| \int F^k_\epsilon(y,z) f_k^j(y) \,dy \right|^2 \right]^{1/2} \right\|_{L^{q'}(z)} & \lesssim \epsilon^{\frac{3}{2}-\frac{1}{2q'}} 2^j \left[ \epsilon^{-\frac{1}{2}-\frac{1}{2q'}} 2^{-2j} \left\| \left( \sum_k (f^j_k)^2 \right)^{1/2}\right\|_{L^p}^2 \epsilon^{-1+\frac{1}{p}} \right]^{1/2} \\
& \lesssim \epsilon^{\frac{3}{4q}+\frac{1}{2p}} \left\| \left( \sum_k (f^j_k)^2 \right)^{1/2}\right\|_{L^p}.
\end{split}
\end{equation*}
As noticed at the end of Step 4, this inequality implies the claim; this concludes the proof.
\end{proof}

\begin{rem} Let us come back to (\ref{mesangebis}). It can be written as 
follows
$$ \left|B_{m_\epsilon}(f,g)(0)\right|\lesssim \sum_{k,l} 2^{2l} 
\alpha_{l} \frac{\epsilon^{3/2}}{2^{2l}}\int_{(\sqrt{\epsilon} 
y,\epsilon z) \in R_{-k\sqrt{\epsilon}}([-2^l,2^l]^2)} |f_k \otimes 
g_k(y,z)| dydz.$$
The set
$$ \{(y,z),\ (\sqrt{\epsilon} y,\epsilon z) \in 
R_{-k\sqrt{\epsilon}}([-1,1]^2)\}$$
is a rectangle of dimensions $2^{l+1} \epsilon^{-1/2}$ and 
$2^{l+1}\epsilon^{-1}$ whose the measure is $2^{2l+2}\epsilon^{-3/2}$. 
As a consequence, we have
\begin{align*}
  \left|B_{m_\epsilon}(f,g)(0)\right| & \lesssim \sum_{k,l} 2^{2l} 
\alpha_{l} {\mathcal K}_{\epsilon^{-1/2}} (f_k \otimes g_k)(0,0) \\
  & \lesssim \sum_{k} {\mathcal K}_{\epsilon^{-1/2}}(f_k \otimes g_k)(0,0),
  \end{align*}
with ${\mathcal K}_{\epsilon^{-1/2}}$ the Kakeya maximal operator on 
$\R^2$ (see Section 10.3 in \cite{Grafakos} for a modern review on this 
subject).
Translating in $x$, we get
$$  \left|B_{m_\epsilon}(f,g)(x)\right|  \lesssim \sum_{k} {\mathcal 
K}_{\epsilon^{-1/2}}(f_k \otimes g_k)(x,x).$$
So the boundedness of $B_{m_\epsilon}$ is closely related to the 
boundedness of a ``bilinear Kakeya operator'' (the one corresponding to 
restrict a $2$-dimensional linear Kakeya operator on the diagonal.
\end{rem}

\section{Interpolation of the different results} \label{sec:combinaison}

\begin{thm} \label{thm:thmgeneral} Let $p,q,r\in [1,\infty]$ be exponents satisfying
$$ 1 \leq \frac{1}{p} + \frac{1}{q} + \frac{1}{r}.$$
Then for all smooth and bounded curve $\Gamma$, we have
$$ \left\| T_{m_\epsilon} \right\|_{L^p \times L^q \rightarrow L^{r'}} \lesssim \epsilon^\rho,$$
where
\be{eq:rho} \rho:=\frac{1}{\max\{p,2\} }+\frac{1}{\max\{q,2\} } +\frac{1}{\max\{r,2\} } -1 \ee
in the following cases:
\begin{itemize}
 \item[a)] if $1\leq \frac{1}{p}+\frac{1}{q}+\frac{1}{r}\leq \frac{3}{2}$ with $p,q,r<\infty$ and $\rho\geq 0$;
 \item[b)] if $ \frac{3}{2}\leq \frac{1}{p}+\frac{1}{q}+\frac{1}{r}$ (with eventual one infinite exponent and in this case, $\rho$ is always non-negative).
\end{itemize}
Moreover, if $p,q,r\leq 2$ the exponent $\rho$ can be improved in $\tilde{\rho}$ given by
$$ \tilde{\rho}:= \min\left\{\frac{1}{p},\frac{1}{q},\frac{1}{r}\right\}.$$
We point out that if at least two of the three indices $(p,q,r)$ are lower than $2$, then $\rho=\tilde{\rho}$.
\end{thm}

\begin{figure}
\psfragscanon
\psfrag{000}[l]{$(0,0,0)$}
\psfrag{100}[l]{$(1,0,0)$}
\psfrag{001}[l]{$(0,0,1)$}
\psfrag{010}[l]{$(0,1,0)$}
\psfrag{222}[l]{$(\frac{1}{2},\frac{1}{2},\frac{1}{2})$}
\psfrag{112}[l]{$(1,1,\frac{1}{2})$}
\psfrag{111}[l]{$(1,1,1)$}
\psfrag{011}[l]{$(0,1,1)$}
\psfrag{211}[l]{$(\frac{1}{2},1,1)$}
\psfrag{221}[l]{$(\frac{1}{2},\frac{1}{2},1)$}
\psfrag{121}[l]{$(1,\frac{1}{2},1)$}
\psfrag{212}[l]{$(\frac{1}{2},1,\frac{1}{2})$}
\psfrag{122}[l]{$(1,\frac{1}{2},\frac{1}{2})$}
\psfrag{C}[l]{$\bar{C}$}
\begin{center}
\resizebox{0.7\textwidth}{!}{\includegraphics{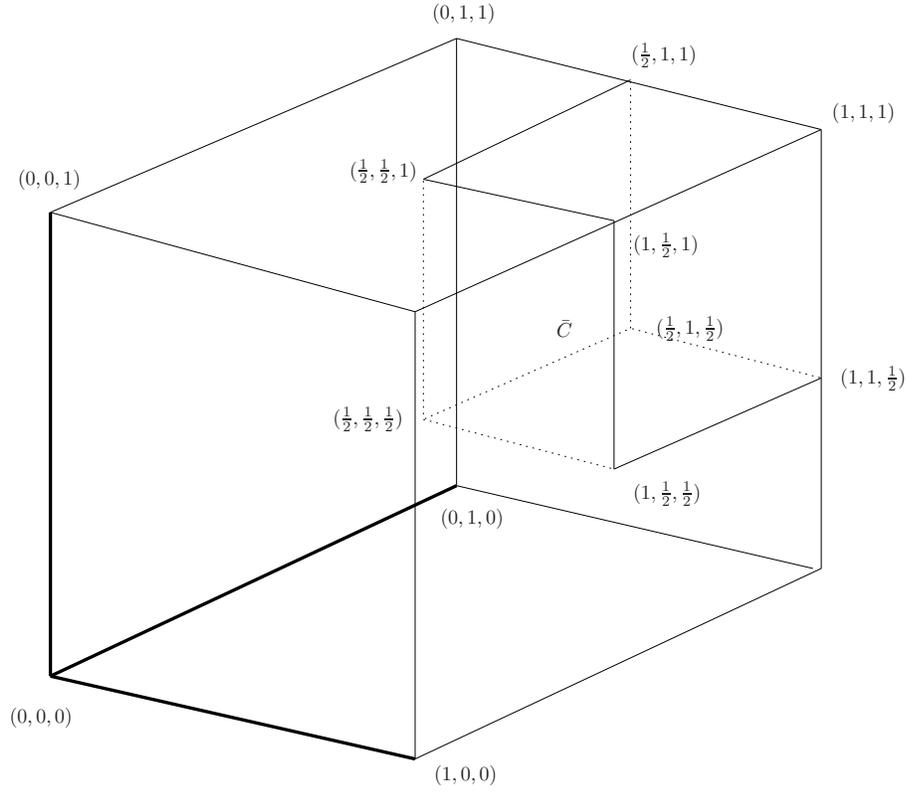}}
\end{center}
\caption{exponents $(p^{-1},q^{-1},r^{-1})\in [0,1]^3$}
\label{figure1}
\end{figure}

\begin{proof}
The case for $p,q$ and $r$ finite was proved in Proposition \ref{prop:L2bis}.  So, let us just consider the case $b-)$ with only one infinite exponent (since we cannot have two infinite exponents), by symmetry $p=\infty$. Then $\frac{1}{q}+\frac{1}{r}\geq \frac{3}{2}$ implies $q,r\leq 2$. Proposition \ref{prop:i12}  proves the result for $(q,r)=(1,2)$ and by symmetry for $(q,r')=(2,1)$ and Proposition \ref{prop:11i} proves the result for $(q,r)=(1,1)$. So by interpolation, we have the desired estimate for all exponents $\frac{1}{q}+\frac{1}{r}\geq \frac{3}{2}$, which ends the proof of case $b)$. \\
Let us now show the last claim about an improvement for $p,q,r\leq 2$. Indeed, from Propositions \ref{prop:L2}, \ref{prop:112} and \ref{prop:221}, we know that the exponent $\rho=1/2$ is optimal on the points $(p,q,r)=(2,2,2)$, $(2,2,1)$, $(1,1,2)$. As a consequence by symmetry and interpolation, we know that we can obtain an exponent $1/2$ as soon as $\max\{p,q,r\}=2$, which corresponds to $\rho$ given by (\ref{eq:rho}). Since Proposition \ref{prop:111} proves that we can have an exponent $1$ at the point $(1,1,1)$, we can interpolate each point $u:=(x,y,z)$ belonging to the cube $C:=[1/2,1]^{3}$ by the end-point $(1,1,1)$ with another point belonging to the subset 
$$\bar{C}:=\{(p^{-1},q^{-1},r^{-1}), \ \max\{p,q,r\}=2\}.$$ Indeed, if $x=\min\{x,y,z\}$ then we have
$$ (x,y,z) = (2x-1)(1,1,1) + 2(z-x)(\frac{1}{2},\frac{1}{2},1) + 2(y-x)(\frac{1}{2},1,\frac{1}{2}) + (2+2x-y-z)(\frac{1}{2},\frac{1}{2},\frac{1}{2}),$$
which by interpolation gives the exponents
$$ \tilde{\rho} = (2x-1) + \frac{1}{2}\left[2(z-x)+2(y-x) + (2+2x-y-z)  \right] = x.$$
\end{proof}

\mb
As we have seen the exponent can be improved if the curve $\Gamma$ is nowhere characteristic.

\begin{thm} Let $p,q,r\in [1,\infty]$ be exponents satisfying
\be{eq:homo} 1 \leq \frac{1}{p} + \frac{1}{q} + \frac{1}{r} \ee
and $\min\{p,q,r\}<2$ (else we are in the local-$L^2$ case, and the above estimates cannot be improved).
Then for all smooth, bounded and nowhere characteristic curve $\Gamma$, we have
$$ \left\| T_{m_\epsilon} \right\|_{L^p \times L^q \rightarrow L^{r'}} \lesssim \epsilon^\rho,$$
where
\be{eq:rho2} \rho:=\min\left\{ \frac{1}{\max\{p,2\} }+\frac{1}{\max\{q,2\} } +\frac{1}{\max\{r,2\} } -1 +\left( \max\{\frac{1}{p},\frac{1}{q},\frac{1}{r}\} - \frac{1}{2}\right) , 1\right\} \ee
in the following cases:
\begin{itemize}
 \item[a)] if $1\leq \frac{1}{p}+\frac{1}{q}+\frac{1}{r}\leq 2$, $p,q,r<\infty$ and in that case, $\rho$ is always non-negative ; 
 \item[b)] if $\frac{1}{p}+\frac{1}{q}+\frac{1}{r}=2$ (with eventually one infinite exponent).
\end{itemize}
Moreover, if $p,q,r\leq 2$ and $\frac{1}{p}+\frac{1}{q}+\frac{1}{r}\geq 2$, then $\rho=1$.
\end{thm}

\begin{proof} For the case $a)$, it is a direct consequence of Proposition \ref{prop:L22} : if $p:=\min\{p,q,r\}$ then we can estimate the operator in $L^p$ and not in $L^2$ since the curve is nowhere characteristic. Indeed if two exponents are lower than $2$, it is obvious that $\rho\geq 0$ and if only one exponent is lower than $2$ then 
$$ \rho = \frac{1}{p}+\frac{1}{q}+\frac{1}{r} -1 \geq 0$$
due to (\ref{eq:homo}). So the improved exponent $\rho$ is always non-negative. \\
Concerning the case $b)$, the point $(1,1,\infty)$ has been studied in Proposition \ref{prop:11i} and so by symmetry and interpolation we get all the points $(p,q,r)$ with $\frac{1}{p}+\frac{1}{q}+\frac{1}{r}=2$. \\
Let us now check the last point when $\frac{1}{p}+\frac{1}{q}+\frac{1}{r}\geq 2$. The extremal points of the corresponding set of exponents are $(1,1,1)$, $(1,2,1)$, $(1,2,2)$ and the symmetrical points. So the result is obtained by interpolation and Propositions \ref{prop:112}, \ref{prop:221} and \ref{prop:11i}.
\end{proof}

\begin{cor} \label{cor:nonc} Assume that $\Gamma$ is nowhere characteristic, and let $m_\epsilon$ belong to $\mathcal{M}^\Gamma_\epsilon$. Then
\begin{itemize} 
\item  If $\left\{ \displaystyle \begin{array}{l} 
1 \leq \frac{1}{p}+\frac{1}{q}+\frac{1}{r} < 2 \vsp \\ 
\frac{1}{r}+\frac{1}{q} \leq \frac{3}{2} \vsp \\ 
\frac{1}{p}+\frac{1}{q} \leq \frac{3}{2} \vsp \\ 
\frac{1}{p}+\frac{1}{r} \leq \frac{3}{2}, \vsp  
\end{array} \right.$, then $\displaystyle \|B_{m_\epsilon}\|_{L^p \times L^q \rightarrow L^{r'}} \lesssim \epsilon^{\frac{1}{p}+\frac{1}{q}+\frac{1}{r}-1}$,
and this exponent of $\epsilon$ is optimal.
\medskip

\item If $\displaystyle \left\{ \begin{array}{l} q \geq 2 \geq p,r  \vsp \\ \frac{1}{p}+\frac{1}{r} \geq \frac{3}{2}, \vsp  \end{array} \right.$, then $\displaystyle
\|B_{m_\epsilon}\|_{L^p \times L^q \rightarrow L^{r'}} \lesssim \epsilon^{\frac{1}{q}+\frac{1}{2}}$.
\medskip

\item The above statement of course remains true if the indices $(p,q,r)$ are permuted.
\end{itemize}
\end{cor}

\begin{proof} Let consider the first point. Since the previous result, we know that the first claim is true for $(p,q,r)$ such that 
$$ 1 \leq \frac{1}{p}+\frac{1}{q}+\frac{1}{r} \leq 2$$
with at most one exponent lower than $2$. This set of exponents is also composed of $4$ sub-squares of length $1/2$. Then interpolating between them, the convex hull of these ones is exactly described by the given inequalities. \\
The second claim is a consequence of the interpolation of the different extremal points : $(p,q,r)=(1,2,2), (2,2,1), (1,t,2),(2,t,1)$ (with $t\rightarrow \infty$ described in the last theorem) and $(p,q,r)=(1,2,1),(1,\infty,1)$ (obtained in Subsections \ref{subsec:112} and \ref{subsec:11i}).
 \end{proof}

\section{Bilinear Fourier transform restriction-extension inequalities} \label{sec:bilinearrest}

\begin{df} Let $p,q,r\in [1,\infty]$ satisfy (\ref{eq:subholder}). We say that a curve $\Gamma \subset \R^2$ satisfies a {\it $(p,q,r)$ restriction-extension inequality} if the frequency restriction-extension bilinear multiplier
$$ T_\Gamma:=(f,g) \rightarrow T_\Gamma(f,g)(x):=\int e^{ix(\xi+\eta)} \widehat{f}(\xi) \widehat{g}(\eta) d\sigma_\Gamma(\xi,\eta),$$
where $d\sigma_\Gamma$ is the arc-length measure on the curve, is bounded from $L^p(\R)\times L^q(\R)$ into $L^{r'}(\R)$.
\end{df}
 
Let us first say a few words concerning the linear theory. For exponents $p,r\in[1,\infty]$ and $\Gamma$ a curve in $\R^2$, we could ask when the linear operator on $\R^2$
$$ U_\Gamma(f):=x \rightarrow \int e^{ix.\xi} \widehat{f}(\xi) d\sigma_\Gamma(\xi)$$
is bounded from $L^p(\R^2)$ into $L^{r'}(\R^2)$. This operator is a multiplier and corresponds to the convolution operation with ``$\widehat{d\sigma_\Gamma}$''. 

\mb From Lemma \ref{lem:kernelosci}, it follows that for a compact smooth curve $\Gamma$, $\widehat{d\sigma_\Gamma}$ belongs to $L^{4+\epsilon}(\R^2)$ for every $\epsilon>0$. By Young inequality, we deduce that the operator $U_\Gamma$ is bounded from $L^p(\R^2)$ to $L^{r'}(\R^2)$ for every exponents $p,r\geq 1$ satisfying
$$ \frac{1}{r'}+1 <\frac{1}{4}+\frac{1}{p},$$
which is equivalent to 
$$ \frac{1}{p}+\frac{1}{r}>\frac{7}{4}.$$

\gb Under the same assumption of non-vanishing curvature, we now want to obtain similar results for the bilinear operator $T_\Gamma$.

\begin{prop}[Bilinear restriction]  \label{prop:bilinearrestriction} Assume that the compact and smooth curve $\Gamma$ has a non vanishing Gaussian curvature. Then, for all exponents $p,q,r\in[1,\infty)$ satisfying
\be{eq:condiexp0} \frac{1}{p}+\frac{1}{q}+\frac{1}{r} > \frac{5}{2}, \ee
we have the bound
$$ \|T_\Gamma(f,g) \|_{L^{r'}} \lesssim \| f\|_{L^p} \|g\|_{L^q},$$
so $\Gamma$ satisfies a $(p,q,r)$ restriction-extension inequality.
\end{prop}

\begin{proof} This follows from Theorem~\ref{thnvc}. An alternative proof is as follows: due to the assumption on the curvature, Lemma \ref{lem:kernelosci}, yields
$$  \left|\int e^{i(x_1\xi+x_2\eta)} d\sigma_\Gamma(\xi,\eta)\right| \lesssim (1+|(x_1,x_2)|)^{-1/2}. $$
Consequently, the bilinear kernel $K$ in $\R^2$ of $T_\Gamma$ belongs to $L^{4+\epsilon}(\R^2)$ for all $\epsilon>0$. Then the result comes from usual estimates for bilinear convolution (Brascamp-Lieb inequality~\cite{BL}~\cite{L}).
\end{proof}

\begin{prop} Assume that the curve $\Gamma \subset \R^2$ satisfies a $(p,q,r)$-restriction inequality with $p,q,r\in (1,\infty)$. Then for all smooth symbol $m\in\s(\R^2)$, there exists a constant $C$ such that 
 $$ \|B_{md\sigma_\Gamma} (f,g) \|_{L^{r'}} \leq C \| f\|_{L^p} \|g\|_{L^q}.$$
\end{prop}

\begin{proof} We just develop the smooth symbol $\sigma$ via Fourier transform: there exists a smooth function $k\in\s(\R^2)$ such that
$$ \sigma(\xi,\eta)=\int e^{i(\xi y + \eta z)} K(y,z) dy dz.$$
So we have
$$ B_{md\sigma_\Gamma} = \int K(y,z) T_\Gamma(\tau_y f,\tau_z g) dy dz.$$ 
Then, using that the translation do not change the Lebesgue norm and $K\in L^1(\R^2)$, it follows by Minkowski inequality that $B_{md\sigma_\Gamma}$ is bounded from $L^p \times L^q$ to $L^{r'}$. \end{proof}

\mb We want now to combine the two previous kinds of argument (using the decay of the kernel due to the curvature and the orthogonality properties in the frequency space, used in Section \ref{SOPP}).

\begin{prop} \label{prop:proprest} Assume that the smooth curve $\Gamma$ has a non vanishing curvature. Then the bilinear multiplier $T_\Gamma$ associated to the singular symbol $m(\xi,\eta):=d\sigma_\Gamma(\xi,\eta)$ is bounded from $L^p(\R) \times L^q(\R)$ into $L^{r'}(\R)$ as soon as $p,q,r\in (1,\infty)$ satisfy
\be{eq:condiexp}
\left\{ \begin{array}{l}
        \frac{1}{p}-\frac{1}{q}+\frac{1}{r}\leq 1 \vsp \\
        \frac{-1}{p}+\frac{1}{q}+\frac{1}{r} \leq 1 \vsp \\
        \frac{1}{p}+\frac{1}{q}-\frac{1}{r}\leq 1 \vsp \\
        \frac{1}{p}+\frac{1}{q}+\frac{1}{r} >\frac{7}{3}. \vsp 
        \end{array} \right.
\ee
So $\Gamma$ satisfies a $(p,q,r)$-restriction inequality for such exponents.
\end{prop}

\begin{proof} The idea is to improve the previous estimates by interpolating with the decay of the kernel (Lemma \ref{lem:kernelosci}). \\
Consider $K$ the linear kernel in $\R^2$ given by
$$ K(x_1,x_2) : = \int e^{i(x_1\xi+x_2\eta)} d\sigma_\Gamma(\xi,\eta).$$
We split the kernel in the space variable, using a function $\Psi \in \mathcal{S}$ (such that $\Psi(0)=0$, $\widehat{\Psi}$ is compactly supported in $[-1,1]$, and $\sum_j \Psi(2^j \cdot)=1$) as follows:
$$ K(x_1,x_2) =\sum_{j\geq 0}  \Psi(2^{-j}(x_1,x_2)) K(x_1,x_2) + \Phi(x_1,x_2) K(x_1,x_2):= \sum_j K_j + K_\phi,$$
where $\Phi$ satisfies
$$ \Phi (\cdot) = 1-\sum_{j\geq 0} \Psi(2^{-j} \cdot).$$
Since $K$ satisfies the bound
$$ \left| K(x_1,x_2)\right| \lesssim (1+|(x_1,x_2)|)^{-1/2},$$
(due to Lemma \ref{lem:kernelosci}), it follows that
$$ \| K_j\|_{L^\infty(\R^2)} \lesssim 2^{-j/2},$$
which gives
\be{cont1}   \|B_ {\widehat{K_j}}\|_{L^1 \times L^1 \to L^\infty } \lesssim 2^{-j/2}. \ee
Moreover since $K\in L^\infty$, it comes
\be{cont1bis}   \|B_ {\widehat{K_\Phi}}\|_{L^1 \times L^1 \to L^\infty } \lesssim 1. \ee
In addition, by writing the kernel $K_j$ in the frequency space, we have
$$ \widehat{K_j}(\xi_1,\eta_1) = \int 2^{2j} \widehat{\Psi}(2^j(\xi_1-\xi_2,\eta_1-\eta_2)) d\sigma_\Gamma(\xi_2,\eta_2).$$
Consequently the bilinear symbol $(\xi,\eta) \rightarrow  2^{-j} \widehat{K_j}(\xi_1,\eta_1)$ belongs to $\mathcal{M}_{2^{-j}}^\Gamma$. \\
According to Subsection \ref{subsec:221}, the bilinear operator associated to $2^{-j} K_j$ is bounded from $L^2 times L^2$ into $L^\infty$ (and by changing the role of $p,q,r$). So, we know that  
\be{cont2} \|B_ {\widehat{K_j}}\|_{L^{p_0} \times L^{q_0} \to L^{r_0'} } \lesssim 2^j 2^{-3j/4} \lesssim 2^{j/4} \ee
for all exponents $p_0,q_0,r_0\in[1,2]$ satisfying
\begin{equation}
 \frac{1}{p_0}+\frac{1}{q_0} + \frac{1}{r_0}=2. \label{eq:hold}
 \end{equation}
Concerning the remainder term $K_\Phi$, it is clear that
\be{cont2bis} \|B_ {\widehat{K_\Phi}}\|_{L^{p_0} \times L^{q_0} \to L^{r_0'} } \lesssim 1. \ee
Using real or complex bilinear interpolation in a one hand between (\ref{cont1}) and (\ref{cont2}) and in the other hand between (\ref{cont1bis}) and (\ref{cont2bis}), it yields that for every ``intermediate'' triplet $(p,q,r)$ between $(p_0,q_0,r_0)$ and $(1,1,1)$ (where $(p_0,q_0,r_0)$ is any triplet of $[1,2]^3$ verifying (\ref{eq:hold}))
\be{eq:decrois} \|B_ {\widehat{K_\Phi}}\|_{L^{p} \times L^{q} \to L^{r'} } \lesssim 1 \quad \textrm{and} \quad \|B_ {\widehat{K_j}}\|_{L^{p} \times L^{q} \to L^{r'} } \lesssim 2^{-\epsilon j}, \ee
with some $\epsilon:=\epsilon(p,q,r)>0$ as soon as $\frac{1}{p}+\frac{1}{q}+\frac{1}{r} >\frac{7}{3}$.
Then, summing over $j\geq 0$ proves the boundedness of $T_K$: $ \|T_ {K}\|_{L^{p} \times L^{q} \to L^{r'} } <\infty$. The range of allowed exponents exactly is the one described by (\ref{eq:condiexp}). Indeed the first inequality in (\ref{eq:condiexp}) is the one given by the plane containing $p=q=r=1$, $p=q=2$ $r=1$ and $p=1$ $q=r=2$, the second inequality is given by the plane containing $p=q=r=1$, $p=q=2$ $r=1$ and $q=1$ $p=r=2$ and the third inequality is given by the plane containing $p=q=r=1$, $r=q=2$ $p=1$ and $q=1$, $p=r=2$. The fourth equation in (\ref{eq:condiexp}) corresponds to the condition required in order to have some $\epsilon>0$ (due to the strict inequality) such that (\ref{eq:decrois}) holds.
\end{proof}

The set of $(p,q,r)$ given by (\ref{eq:condiexp}) is the tetrahedron built on the points $(1,1,1)$, $(1,3/2,3/2)$, $(3/2,1,3/2)$ and $(3/2,3/2,1)$. So by interpolation with Proposition \ref{prop:bilinearrestriction}, we get the following result.

\begin{thm} \label{thm:rest} Assume that the smooth curve $\Gamma$ has a non vanishing curvature. Then the bilinear multiplier $T_\Gamma$ associated to the singular symbol $m(\xi,\eta):=d\sigma_\Gamma(\xi,\eta)$ is bounded from $L^p(\R) \times L^q(\R)$ into $L^{r'}(\R)$ as soon as $(p,q,r)\in [1,2]^3$ belongs to the convex hull of the seven following points~: 
\be{eq:condiexpbbis} (1,1,1),\ (1,1,2), (3/2,3/2,1) \ee
and the others ones obtained by symmetry.
So $\Gamma$ satisfies a $(p,q,r)$-restriction inequality for such exponents.
\end{thm}

\mb Having obtained some ``bilinear Fourier restriction-extension inequalities'', we now come back to the smooth symbols $m_\epsilon$. For a curve $\Gamma$, it should be reasonable from a bilinear Fourier restriction-extension inequality to prove (\ref{ecureuil}) with a decay function $\alpha(\epsilon)=\epsilon$. However, to do that, we have to decompose the $\epsilon$-neighborhood at the scale $\epsilon$ and then to sum up all these small pieces. Since we start from a global estimate with a symbol carried on the whole curve, we have to do this splitting uniformly ``along the tangential variable''. That is why we cannot deduce (\ref{ecureuil}) for all symbols $m_\epsilon$ belonging to the class ${\mathcal M}^\Gamma_\epsilon$. 

\mb Let us assume that $\Gamma$ is a smooth and compact curve and denote the distance function $\nu:=d_\Gamma$. For every $(\xi,\eta)\notin \Gamma$, we know that $|\nabla \nu(\xi,\eta)|=1$. With this notation, $\nabla \nu$ can be considered as the direction of local normal coordinates and $(\nabla \nu)^{\perp}$ as the direction of the local tangential coordinate. \\
We are interesting in symbols $m_\epsilon$ satisfying a nice behavior in the tangential directions given by $\nabla \nu^\perp$. More precisely, we are interested with the symbols $m_\epsilon$ taking the following form
\begin{equation}
 m_\epsilon := \frac{1}{\epsilon} \int_\Gamma  \phi \left(\frac{(\xi,\eta)-\lambda}{\epsilon}\right) m(\lambda) d\sigma_\Gamma(\lambda) , \label{eq:symbconv} 
\end{equation}
with a smooth and compactly supported function $m$ on $\Gamma$ and a smooth function $\phi\in C^\infty_0(\R^2) $ such that 
$ \phi(\xi,\eta)=1 $ if $|(\xi,\eta)|\leq 1/2$ and $ \phi(\xi,\eta)=0 $ if $|(\xi,\eta)|\geq 1$.

Let us check the following point ``$m_\epsilon$ is regular at the scale $\epsilon$ in the normal direction and at the scale $1$ in the tangential direction''. 

\begin{prop} \label{prop:resttomul2bis} Let $m_\epsilon$ be a symbol given by (\ref{eq:symbconv}), then it satisfies the following regularity~:
 \be{eq:regs} \left\| \nabla^\alpha m_\epsilon \right\|_{L^\infty(\R^2)} \lesssim \epsilon^{-|\alpha|},\quad \left\| \langle (\nabla \nu)^{\perp} ,\nabla m_\epsilon \rangle \right\|_{L^\infty(\R^2)} \lesssim 1.\ee
\end{prop}

\begin{proof} First for fixed $(\xi,\eta)$, 
$$ \left|m_\epsilon(\xi,\eta)\right| \leq \epsilon^{-1} {\mathcal H}^1(\Gamma \cap B((\xi,\eta),\epsilon)) \lesssim 1$$
and
$$ \nabla m_\epsilon(\xi,\eta) =  \frac{1}{\epsilon^2} \int_\Gamma  \nabla \phi \left(\frac{(\xi,\eta)-\lambda}{\epsilon}\right) m(\lambda) d\sigma_\Gamma(\lambda).$$
By iterating, we easily get that $\|\nabla^\alpha m_\epsilon\|_{L^\infty(\R^2)}\lesssim \epsilon^{-|\alpha|}$. Let us check the tangential derivative. Let us choose $\gamma$ a normalized parametrization of $\Gamma$ : $\gamma:[0,1] \rightarrow \R^2$ with $|\gamma'(t)|=1$. Then, 
$$ \nabla m_\epsilon(\xi,\eta) =  \frac{1}{\epsilon^2} \int_0^1  \nabla \phi \left(\frac{(\xi,\eta)-\gamma(t)}{\epsilon}\right) m(\gamma(t))dt.$$
Since
\begin{align*}
 \frac{1}{\epsilon} \int_0^1 \langle \nabla \phi \left(\frac{(\xi,\eta)-\gamma(t)}{\epsilon} \right), \gamma'(t) \rangle m(\gamma(t)) dt & = -\int_0^1 \left[\frac{d}{dt} \phi \left(\frac{(\xi,\eta)-\gamma(t)}{\epsilon} \right)\right] m(\gamma(t)) dt \\
 & = \int_0^1 \phi \left(\frac{(\xi,\eta)-\gamma(t)}{\epsilon} \right) \left[\frac{d}{dt} m(\gamma(t))\right] dt.
 \end{align*}
So as previously, 
$$ \left| \frac{1}{\epsilon^2} \int_0^1 \langle \nabla \phi \left(\frac{(\xi,\eta)-\gamma(t)}{\epsilon} \right), \gamma'(t) \rangle m(\gamma(t)) dt \right| \lesssim 1.$$
Consequently, 
\begin{align*}
\left| \frac{1}{\epsilon^2} \int_0^1 \langle \nabla \phi \left(\frac{(\xi,\eta)-\gamma(t)}{\epsilon} \right),  (\nabla \nu)^{\perp} \rangle m(\gamma(t)) dt \right| & \\
& \hspace{-3cm} \lesssim 1 + \left| \frac{1}{\epsilon^2} \int_0^1 \langle \nabla \phi \left(\frac{(\xi,\eta)-\gamma(t)}{\epsilon} \right),  \gamma'(t)^\perp\rangle \langle \gamma'(t)^\perp, (\nabla \nu)^{\perp} \rangle m(\gamma(t)) dt \right| .\end{align*}
However $\langle \gamma'(t)^\perp, (\nabla \nu)^{\perp} \rangle= \langle \gamma'(t), \nabla \nu \rangle$ and since $|(\xi,\eta)-\gamma(t)|\leq \epsilon$, the smoothness of the curve $\Gamma$ implies that 
$$ \left|\langle \gamma'(t), \nabla \nu \rangle \right| \lesssim \epsilon,$$
which concludes the proof of (\ref{eq:regs}).
\end{proof} 

By similar arguments, we can obtain estimates for the higher order differentiation of these specific symbols.

\begin{cor} The symbols $m_\epsilon$ given by (\ref{eq:symbconv}) belong to the class ${\mathcal N}^\Gamma_\epsilon$.
\end{cor}

For such specific symbols $m_\epsilon$, we can prove equivalence for every exponents $p,q,r$ between a restriction-extension property and a decay rate in (\ref{ecureuil}) with $\alpha(\epsilon)=\epsilon$.

\begin{prop} \label{prop:resttomul2}
Let $\Gamma$ be a smooth and compact curve. Assume that for some exponents $p,q,r'\in[1,\infty]$, the compact curve $\Gamma$ satisfies a $(p,q,r)$-restriction inequality. 
Then there exists a constant $c$ such that for all $\epsilon\leq 1$ and all symbol $m_\epsilon$ given by (\ref{eq:symbconv})
\be{eq:op2bis} \|T_{m_\epsilon}\|_{L^p \times L^q \to L^{r'}} \leq c\epsilon. \ee
\end{prop}

\begin{proof} By a change of variables, it comes
\begin{align*} 
T_{m_\epsilon}(f,g)(x)& := \epsilon^{-1} \int_{\Gamma}\int_{\R^2} e^{ix(\xi+\eta)} \phi \left(\frac{(\xi,\eta)-\lambda}{\epsilon}\right) m(\lambda) \widehat{f}(\xi) \widehat{g}(\eta) \, d\xi d\eta d\sigma_\Gamma(\lambda) \\
 & = \epsilon^{-1} \int_{\Gamma}\int_{\R^2} e^{ix(\xi+\eta+\lambda_1+\lambda_2)} \phi \left(\frac{(\xi,\eta)}{\epsilon}\right) m(\lambda) \widehat{f}(\xi+\lambda_1) \widehat{g}(\eta+\lambda_2) \, d\xi d\eta d\sigma_\Gamma(\lambda).
 \end{align*}
Hence, 
$$ T_{m_\epsilon}(f,g)(x) = \epsilon^{-1} \int_{\R^2} \phi \left(\frac{(\xi,\eta)}{\epsilon}\right) e^{ix(\xi+\eta)} T_{md\sigma_\Gamma}(M_\eta f, M_\xi g)(x) \, d\xi d\eta,$$
where $M_\xi$ and $M_\eta$ are modulation operators. Since $T_{d\sigma_\Gamma}$ is assumed to be bounded and $m$ is smooth (at the scale $1$) then $T_{md\sigma_\Gamma}$ is still bounded from $L^p \times L^q$ into $L^{r'}$. Minkowski inequality concludes the proof.
\end{proof}

\section{Applications to bilinear multipliers}

\subsection{Non-smooth bilinear multipliers and bilinear Bochner-Riesz means} \label{subsec:bochnerriesz}

This subsection is devoted to the following question: whether characteristic functions of a compact set of $\R^2$ give a bilinear multiplier bounded from $L^p(\R) \times L^q(\R)$ into $L^{r'}(\R^2)$. We refer the reader to the introduction for a presentation of works concerning the ball and polygons.

We only give a sample of results in this direction, but do not aim at exhaustiveness.

\begin{thm} \label{thm:tt} Let $K$ be a compact subset of $\R^2$
\begin{itemize}
\item[(i)]  If $\partial K$ is an Ahlfors regular curve in $\R^2$ with ``finitely bi-Lipschitz projections'', then for exponents $p,q,r'\in [2,\infty)$,
\be{eq:boundedness} \left\|B_{\chi_{K}}(f,g) \right\|_{L^{r'}} \lesssim \|f\|_{L^p} \| g\|_{L^q} \ee
as soon as 
$$ \frac{1}{p}+\frac{1}{q}+\frac{1}{r} > 1.$$
If $p,q,r' >1$, then (\ref{eq:boundedness}) still holds if
$$ \frac{1}{\max\{p,2\}}+\frac{1}{\max\{q,2\}}+\frac{1}{\max\{r,2\}} > 1.$$
\item[(ii)] If $\partial K$ is smooth, and has a non-vanishing curvature, then for exponents $p,q,r'\in (1,\infty)$,
$$ \left\|B_{\chi_{K}}(f,g) \right\|_{L^{r'}} \lesssim \|f\|_{L^p} \| g\|_{L^q} $$
as soon as 
$$ \frac{1}{p}+\frac{1}{q}+\frac{1}{r} > 1.$$

\end{itemize}
\end{thm}

\begin{proof} We simply explain the proof of the first statement. Let us denote $\Gamma = \partial K$. Without loss of generality and just for convenience, we assume that the diameter of $K$ is less than one. 

Next we need a partition of unity $(\chi_n)_{n\geq 0}$ such that
\begin{itemize}
\item  we have the decomposition for all $(\xi,\eta)$ in $K$
$$ 1 = \sum_{n\geq 0} \chi_n(\xi,\eta)$$
\item for each integer $n\geq 0$, $\chi_n$ is supported in $\Gamma_{10 2^{-n}}$
\item  for each integer $n\geq 0$ and every multi-index $\alpha$, 
$$ \|D^\alpha \chi_n\|_{L^\infty(\R^2)} \lesssim 2^{n|\alpha|}.$$
\end{itemize}
To build this decomposition, consider a Whitney covering of $K^c$ by balls $(O_i=B(x_i,r_i))_i$. Then we can consider a smooth adapted partition of unity $\chi_{O_i}$ and set
$$ \chi_n := \sum_{2^{-n} \leq d((\xi,\eta),S)< 2^{-n+1}} \chi_{O_i}.$$
We let the reader check that these functions satisfy the expected properties due to the notion of Whitney balls. \\

Then
$$ B_{\chi_K}(f,g) = \sum_{n\geq 0} B_{\chi_n}(f,g).$$
In addition, the symbol $m_n$ belongs to $\mathcal{M}_\epsilon^\Gamma$ with $\epsilon=2^{-n}$. Hence, by Proposition \ref{prop:L2} (with Proposition \ref{prop:extension}), we have
$$\left\|T_{m_n}(f,g) \right\|_{L^{r'}} \leq C 2^{-2s n} \|f\|_{L^p} \| g\|_{L^q}.$$
Since $s>0$, we can sum with $n\in\N$ and we finish the proof. The second claim is obtained by the same reasoning with Proposition \ref{prop:L2bis}. 
\end{proof}

For example, we can consider $K$ being a disc, a square or any polygon. In the specific case of a disc, Grafakos and Li have obtained in \cite{GL} boundedness in the local-$L^2$ case for the bilinear multiplier under the H\"older scaling. Here we have general results for general sets but in the sub-H\"older scaling.

When the exponents $p,q,r$ satisfy H\"older relation
\be{eq:hol} 1=\frac{1}{r}+\frac{1}{p}+\frac{1}{q} \ee
Proposition \ref{prop:L2} allows us to get estimate for the bilinear multipliers $m_\epsilon$ without decay. So we cannot sum over $n\geq 0$ the different inequalities. To get around this difficulty, we can proceed as for the Bochner-Riesz means. 
Let us recall this phenomenon in the linear setting~: the linear multiplier 
$$ T:=f \rightarrow \widehat{ {\bf 1}_{\overline{B(0,1)}} \widehat{f}}^{-1}$$
is bounded in $L^2(\R^d)$ for every integer $d\geq 1$ and is not bounded in $L^p(\R^2)$ for every $p\neq 2$. This is a famous result of Fefferman, see \cite{Fefferman}. In order to remedy this unboundedness, a possibility is to add some regularity near the boundary and so to study the following linear operator
$$ T^\lambda(f)(x):=\int_{\R^n} e^{ix.\xi} (1-|\xi|^2)_+^\lambda \widehat{f}(\xi) d\xi$$
where
$$ (1-|\xi|^2)_+ : =(1-|\xi|^2){\bf 1}_{|\xi|\leq 1}.$$
This new symbol corresponds to a regularization of the initial symbol ${\bf 1}_{\overline{B(0,1)}}$ at the boundary. \\
Note that $T^\lambda(f)$ converges to $T(f)$ for $\lambda>0$ goes to $0$. However, the symbol in $T^\lambda$ is a little more regular at the boundary ${\mathbb S}^{d-1}$. The main question relies on the range of exponent $p$ (depending on $\lambda$ and $d$) on which $T^\lambda$ is $L^p(\R^d)$-bounded. The question remains open for any dimension $d$ but is completely solved for $d=2$~:

\begin{thm}[\cite{CS}] For $d=2$ and $\lambda \in (0,1/2]$, $T^\lambda$ is $L^p(\R^2)$-bounded if and only if
 $$ \frac{4}{3+2\lambda}<p<\frac{4}{1-2\lambda}.$$
\end{thm}

 We refer the reader to Section 10.2 in \cite{Grafakos} for a modern review of this subject and point out this main idea : add regularity on the characteristic function at the boundary of the set in order to gain integrability of the multiplier.

 We aim to apply this same idea in our bilinear and current setting. So consider a set $K \subset \R^2$ as in Theorem \ref{thm:tt}. The bilinear multiplier associated to the symbol ${\bf 1}_K$ may be not bounded from $L^p \times L^q$ to $L^r$, so we regularize this symbol at the boundary $\partial K$ to get boundedness.

\begin{thm} Let $K$ be a compact set, set
$$
m(\xi,\eta) := {\chi}_K(\xi,\eta) d((\xi,\eta),K)^\lambda
$$
with $\lambda>0$. and let $(p,q,r)$ satisfy (\ref{eq:hol}).
\begin{itemize}
\item[(i)] If $\partial K$ is an Ahlfors regular curve in $\R^2$ with ``finitely bi-Lipschitz projections'', then for exponents $p,q,r\in [2,\infty)$ satisfying (\ref{eq:hol}),
\be{eq:boundedness2} \left\|T_{m}(f,g) \right\|_{L^{r'}} \leq C \|f\|_{L^p} \| g\|_{L^q} \ee
\item[(ii)] If $\partial K$ is a smooth curve with non-vanishing curvature, then for exponents $p,q,r\in (1,\infty)$
$$
\left\|T_{m}(f,g) \right\|_{L^{r'}} \leq C \|f\|_{L^p} \| g\|_{L^q}.
$$
\end{itemize}
\end{thm}

\begin{proof} We only prove the first assertion. As previously, we can decompose the symbol $m_\lambda$
$$ m =\sum_{n\geq 0} m_n$$
with $m_n$ a symbol supported in 
$$ K \cap\{(\xi,\eta),\ 2^{-n-1} \leq d((\xi,\eta), \partial K) \leq 2^{-n+2}\}$$
and for all integer $d\geq 0$
 $$ \| \nabla^d m_n \|_{L^\infty(\R^2)} \lesssim 2^{-\lambda n} 2^{-dn}.$$
As a consequence, we deduce from Proposition \ref{prop:L2} that $T_{m_n}$ is bounded from $L^p \times L^q$ to $L^{r'}$
with
$$ \|T_{m_n}\|_{L^p \times L^q \to L^{r'}} \lesssim 2^{-\lambda n}$$
and so $T$ is bounded by summing with $n\geq 0$.
\end{proof}

We let the reader to obtain the other boundedness (with taking other exponents $p,q,r$) according to the geometrical assumptions of the curve $\Gamma$.

\subsection{Singular symbols}

\label{singularsymbols}

Proceeding pretty much as in Section~\ref{subsec:bochnerriesz}, one obtains the following theorem.

\begin{thm}
Let $\Phi$ be a smooth compactly supported function, $\Gamma$ a smooth curve, and let
$$
m(\eta,\xi) = \Phi(\eta,\xi) \operatorname{dist}((\eta,\xi),\Gamma)^{-\alpha}
$$
with $\alpha>0$. Suppose $2<p,q,r<\infty$, and
$$
\alpha < \frac{1}{p} + \frac{1}{q} + \frac{1}{r} - 1.
$$
Then $B_m$ is bounded from $L^p \times L^q$ to $L^{r'}$.
\end{thm}

We could of course obtain corresponding statements for the whole range of exponents $(p,q,r)$, with conditions on $\alpha$ depending on the properties of $\Gamma$. We chose to present only the case $p,q,r>2$ for the sake of simplicity.

\subsection{Bilinear oscillatory integral near a singular domain}

Let $\phi:\R^2 \rightarrow \R$  be a smooth function, $m$ be a smooth, compactly supported symbol, and consider the following bilinear oscillatory integral, 
$$ B_t(f,g)(x):= \int_0^t \int_{\R^2} e^{is\phi(\xi,\eta)}  \widehat{f(s,\cdot)}(\xi) \widehat{g(s,\cdot)}(\eta) m(\xi,\eta) d\xi d\eta ds,$$
where $f$ and $g$ are functions on $\R^+ \times \R$. Some multilinear integrals appear in the space-time resonances method, as explained in \cite{Pierre,BG} (see Subsection \ref{subsec:PDEs}).

The integration over $s$ of the gives
$$\int_0^t  e^{is\phi(\xi,\eta)} ds = 
\left\{ \begin{array}{ll}
\frac{e^{it\phi(\xi,\eta)}-1}{\phi(\xi,\eta)} & \quad\textrm{if  } \phi(\xi,\eta)\neq 0 \dsp \vsp \\
t & \quad\textrm{if  } \phi(\xi,\eta)= 0 . \vsp
\end{array}
\right.
$$
So let us consider the singular set 
$$ S:=\phi^ {-1}(0):=\left\{(\xi,\eta),\ \phi(\xi,\eta)=0 \right\}$$
and assume that $\nabla \phi$ is not vanishing on $\Gamma$, in order that $S$ is a smooth sub-manifold of dimension $1$.

\mb
{\bf Assumption :}
Let us assume that for some exponents $p,q,r$, there exists $\rho\in(0,1]$ such that for all small enough parameter $\epsilon$ then
\be{rho} \left\| T_{m_\epsilon}(f,g) \right\|_{L^{r'}} \lesssim \epsilon^\rho \|f\|_{L^p} \|g\|_{L^q} \ee
as soon as $m_\epsilon$ is a symbol in $\mathcal{M}^S_\epsilon$ or $\mathcal{N}^S_\epsilon$.

Then we have the following result
\begin{prop} Assume that the smooth symbol $m$ is supported on $S_\epsilon$ for $\epsilon\leq 1$.  
Then, the operator $B_t$ is uniformly bounded (with respect to $\epsilon$ and $T$) from $L^\infty_T L^p \times L^\infty_T L^q$ into $L^\infty_T L^{r'}$ as soon as
$$ T\epsilon^\rho \lesssim 1.$$
\end{prop}

\begin{proof} Using a partition of unity associated to $S$, and covering $S_\epsilon$ as in the proof of~\ref{thm:tt} (for $2^{-n}\leq \epsilon$), define for $s\in[0,T]$ $B_t^n$ the bilinear multiplier with symbol
$$\sigma_n = e^{is\phi(\xi,\eta)} \chi_{n}(\xi,\eta) m(\xi,\eta),$$
satisfying for all multi-index $\alpha$:
$$ \|D^\alpha \sigma_n\|_{L^\infty(\R^2)} \lesssim 2^ {n|\alpha|}.  $$
We let the reader check this estimate. Indeed, differentiation may make appear quantities bounded by $s^b 2^{nc}$ with $b+c \leq |\alpha|$ and in this case, we use that $s\leq T\leq \epsilon^{-\rho} \leq 2^{-n\rho} \leq 2^{-n}$, since $\rho\leq 1$.\\
So by assumption, we know that $B_s^n$ is bounded from $L^p \times L^q$ to $L^{r'}$ with
\begin{align*}
 \left\| B_s^n (f(s,\cdot),g(s,\cdot)) \right\|_{L^{r'}} & \lesssim 2^{-\rho n} \|f(s,\cdot)\|_{L^p} \|g(s,\cdot)\|_{L^q} \\
 & \lesssim 2^{-\rho n} \|f\|_{L^\infty_T L^p} \|g\|_{L^\infty_T L^q}.
 \end{align*}
Then we can sum for $n\geq 0$ with $2^{-n}\lesssim \epsilon$ and we get for all $t\in[0,T]$
\begin{align*}
 \left\| B_t (f,g) \right\|_{L^{r'}} & \lesssim \int_0^s \left\|\int_{\R^2} e^{is\phi(\xi,\eta)}  \widehat{f(s,\cdot)}(\xi) \widehat{g(s,\cdot)}(\eta) m(\xi,\eta) d\xi d\eta \right\|_{L^{r'}} ds \\
 & \lesssim T\epsilon^\rho \|f\|_{L^\infty_T L^p} \|g\|_{L^\infty_T L^q}.
\end{align*}
This estimate is uniform with respect to $t\in[0,T]$, so by taking the supremum over $t$ we conclude the proof.
 \end{proof}

\begin{cor} \label{cor2} The bilinear multipliers 
$$ (f,g) \rightarrow \int_0^t \int_{|\phi|\leq \epsilon} e^{is\phi(\xi,\eta)}  \widehat{f(s,\cdot)}(\xi) \widehat{g(s,\cdot)}(\eta) m(\xi,\eta) d\xi d\eta ds $$
are uniformly bounded from $L^\infty_T L^p \times L^\infty_T L^q$ into $L^\infty_T L^{r'}$ as soon as $T\epsilon^\rho\lesssim 1$.
\end{cor}

\begin{cor} \label{cor} In particular for time-independent functions $f,g$, we get that the bilinear multipliers
$$ (f,g) \rightarrow \int_0^t \int_{|\phi|\leq \epsilon} e^{is\phi(\xi,\eta)}  \widehat{f}(\xi) \widehat{g}(\eta) m(\xi,\eta) d\xi d\eta ds $$
are uniformly bounded from $L^p \times L^q$ into $L^\infty_T L^{r'}$ as soon as $T\epsilon^\rho\lesssim 1$.
\end{cor}

\begin{ex} We would like to describe an example in the linear theory for showing that under assumption (\ref{rho}), we cannot expect a better result than the previous corollary. \\
So consider the function $\phi(\xi)=|\xi|^2$ and the corresponding linear operator
$$ T_{\epsilon,t}:= f \rightarrow \int_0^t \int_{|\phi|\leq \epsilon} e^{is\phi(\xi)}  \widehat{f}(\xi) m(\xi) d\xi ds.$$
Let us see when $T$ remains bounded as $t\rightarrow \infty$ and $\epsilon \to 0$. So consider a smooth function $f$ then by a change of variables, it comes
\begin{align*}
 T_{\epsilon,t}(f)(x) & = \int_{|\eta|\leq \sqrt{t}\epsilon} e^{ix\eta/\sqrt{t}} \frac{1-e^{i\eta^2}}{\eta^2} t \widehat{f}(\eta/\sqrt{t}) m(\xi/\sqrt{t}) \frac{d\eta}{\sqrt{t}} \\
 & \simeq_{\genfrac{}{}{0pt}{}{t\to \infty}{\epsilon \to 0}} t\epsilon \widehat{f}(0) m(0).
 \end{align*}
So the limit  can be defined only in $L^\infty$ as soon as $t\epsilon$ is bounded and $f\in L^1$ (to give a sense to $\widehat{f}(0)$).
However let us now see when assumption (\ref{rho}) is satisfied in this particular setting.
So we consider a smooth symbol $m_\epsilon$ at the scale $\epsilon$ of $\phi^{-1}(0)=\{0\}$ and we estimate the multiplier
$$ m_\epsilon(D)(f)(x):= \int e^{ix\xi} \widehat{f}(\xi) m_\epsilon(\xi) d\xi.$$
Then, it comes that 
$$ \left|  m_\epsilon(D)(f)(x) \right| \lesssim \int |f(y)| \frac{\epsilon}{(1+\epsilon|x-y|)^M} dy$$
for every large enough integer $M$. Consequently we get that $ m_\epsilon(D)$ is bounded from $L^1$ to $L^\infty$ with a bound controlled by $\epsilon$. So in this case, assumption (\ref{rho}) is satisfied for $\rho=1$ with $L^1\rightarrow L^\infty$ and we have checked that we can not expect a better result than the one described in Corollary \ref{cor}. \\
We could work with other spaces. For example, the previous computation gives that $\pi_{m_\epsilon}$ is bounded from $L^2$ to $L^\infty$ with a $\epsilon^{1/2}$-bound. In this case, we have to bound the operator $T_{\epsilon,t}$ with the $L^2$-norm of $f$, which can be done as follows
\begin{align*}
 \|T_{\epsilon,t}(f)\|_\infty & \leq \left\|\int_{|\xi|\leq \epsilon} e^{ix\xi} \frac{1-e^{it\xi^2}}{\xi^2}  \widehat{f}(\xi) m(\xi) d\xi \right\|_{L^\infty(x)}\\
 & \lesssim \|f \|_{L^2}  \left(\int_{|\xi|\leq \epsilon}  \left|\frac{1-e^{it\xi^2}}{\xi^2}\right|^2  d\xi\right)^{1/2} \\
 & \lesssim t\epsilon^ {1/2} \|f \|_{L^2}.
 \end{align*}
Moreover, the previous inequalities can be indeed equivalent for some specific choices of $m$ and $f$.
So we recover that
$$ \|T_{\epsilon,t}\|_{L^2\rightarrow L^\infty} \simeq t \epsilon^{1/2}.$$
So one more time, we cannot obtain a better decay in $\epsilon$ than the one described in Assumption \ref{rho}.
\end{ex}

\end{document}